\documentclass[a4paper]{amsart}
\usepackage{amssymb,cmap,verbatim}
\usepackage{stmaryrd} 

\newtheorem{thm}{Theorem}[section]
\newtheorem{lemma}[thm]{Lemma}
\newtheorem{cor}[thm]{Corollary}
\newtheorem{claim}{Claim}[thm]

\newtheorem{fact}[thm]{Fact}
\newtheorem*{thmaa}{Theorem~A}
\newtheorem*{thmb}{Theorem~B}

\theoremstyle{definition}
\newtheorem{defn}[thm]{Definition}

\theoremstyle{remark}
\newtheorem{remark}[thm]{Remark}

\DeclareMathOperator{\acc}{acc}

\DeclareMathOperator{\pr}{pr}

\DeclareMathOperator{\dl}{Dl}

\newcommand{\s}{\subseteq}

\renewcommand{\mid}{\mathrel{|}\allowbreak}
\renewcommand{\restriction}{\mathbin\upharpoonright}

\newcommand{\redub}{\hookrightarrow_B}                
\newcommand{\reduc}{\hookrightarrow_c}                

\title{on Unsuperstable theories in GDST}

\author{Miguel Moreno}
\address{Department of Mathematics, Bar-Ilan University, Ramat-Gan 5290002, Israel.}
\urladdr{http://u.math.biu.ac.il/\textasciitilde morenom3}
\address{Institute of Mathematics, University of Vienna, Vienna 1090, Austria.}
\urladdr{http://miguelmath.com}

\begin{document}
\begin{abstract}
We study the $\kappa$-Borel-reducibility of isomorphism relations of
complete first order theories by using coloured trees. Under some cardinality assumptions, we show the following: For all theories T and T', if T is classifiable
and T' is unsuperstable, then the isomorphism of models of T' is strictly above
the isomorphism of models of T with respect to $\kappa$-Borel-reducibility.
\end{abstract}


\maketitle

\section{Introduction}

The interaction between Generalized Descriptive Set Theory (GDST) and Classification theory has been one of the biggest motivation to study the Borel reducibility in the Generalized Baire spaces. One of the main questions is to determined if there is a counterpart of Shelah's Main Gap Theorem in the Generalized Baire Spaces (provable in ZFC). In \cite{HKM} Hyttinen, Kulikov, and Moreno showed the consistency of a counterpart of Shelah's Main Gap Theorem in the Borel reducibility hierarchy of the isomorphism relations (see preliminaries), indeed it can be forced.

\begin{fact}[Hyttinen-Kulikov-Moreno, \cite{HKM} Theorem 7]
Suppose that $\kappa=\kappa^{<\kappa}=\lambda^+$, $2^\lambda>2^\omega$ and $\lambda^{<\lambda}=\lambda$. There is a forcing notion $\mathbb{P}$ which forces the following statement:

``If $T_1$ is a classifiable theory and $T_2$ is not, then the isomorphism relation of $T_1$ is Borel reducible to the isomorphism relation of $T_2$, and there are $2^\kappa$ equivalence relations strictly between them"
\end{fact}

In the same article the authors proved the following in ZFC.

\begin{fact}[Hyttinen-Kulikov-Moreno, \cite{HKM} Corollary 2]\label{miangapstrict}
Suppose that $\kappa=\kappa^{<\kappa}=\lambda^+$ and $\lambda^{\omega}=\lambda$. If $T_1$ is classifiable and $T_2$ is stable unsuperstable, then the isomorphism relation of $T_1$ is Borel reducible to the isomorphism relation of $T_2$. 
\end{fact}

In this article we will extend Fact \ref{miangapstrict} to unsuperstable theories, i.e. the unstable case.

\begin{thmaa}\label{Unsupred}
Suppose that $\kappa=\kappa^{<\kappa}=\lambda^+$ is such that $\lambda^\omega=\lambda$. If $T_1$ is classifiable and $T_2$ is unsuperstable, then the isomorphism relation of $T_1$ is Borel reducible to the isomorphism relation of $T_2$. 
\end{thmaa}

To prove Theorem A we will use the coloured trees tools developed in \cite{HK} by Hyttinen and Kulikov and the tools used by Shelah in \cite{Sh}, to construct models of unsuperstable theories. In \cite{HK} Hyttinen and Kulikov used the coloured trees to construct models of an already fix stable unsuperstable theory in the context of the Generalized Baire spaces. In \cite{Sh} Shelah used ordered trees with $\omega+1$ levels to construct non-isomorphic models of unsuperstable theories. 

The objective of Hyttinen and Kulikov was to use elements of $\kappa^\kappa$ to construct models of a very particular theory. The difficulties with this construction comes when we want to applied it to unstable theories. On the other hand the objective of Shelah was to use stationary sets to construct as many models as possible for unsuperstable theories. Even though for each unsuperstable theory and $f\in 2^\kappa$, Shelah's constructs a model $\mathcal{M}_f$, this construction does not define a Borel reduction. The problem comes when the ordered trees are constructed.  
We will use coloured trees to construct ordered trees, by doing this we ensure that the construction of the models will define a continuous reduction. To construct the ordered trees from coloured trees we will use similar ideas to ones used by  Abraham in \cite{Av} to construct a rigid Aronszajn tree.

In \cite{FMR} Fernandes, Moreno, and Rinot showed that the isomorphism relation of unsuperstable theories can be forced to be analytically complete for $\kappa$ a successor cardinal. We will extend this result to inaccessible cardinals.

\begin{thmb}\label{Unsupcomp}
Suppose that $\kappa=\kappa^{<\kappa}$ is an inaccessible cardinal. There exists a $<\kappa$-closed $\kappa^+$-cc forcing extension in which: If $T$ is unsuperstable, then the isomorphism relation of $T$ is analytically complete. 
\end{thmb}

\subsection{Preliminaries}
During this paper we will work under the general assumption that $\kappa$ is a regular uncountable cardinal that satisfies $\kappa=\kappa^{<\kappa}$ and for all $\gamma<\kappa$, $\gamma^\omega<\kappa$. We will work only with first-order countable complete theories on a countable language, unless something else is stated.

Let us recall some definitions and results on GDST, for more on GDST see \cite{FHK13}. We will only review the definitions and results that are relevant for the article.


The generalized Baire space is the set $\kappa^\kappa$ endowed with the bounded topology, in this topology the basic open sets are of the form $$[\zeta]=\{\eta\in \kappa^\kappa\mid \zeta\subseteq \eta\}$$
where $\zeta\in \kappa^{<\kappa}$. The collection of $\kappa$-Borel subsets of $\kappa^\kappa$ is the smallest set that contains the basic open sets and is closed under union and intersection both of length $\kappa$. 
A $\kappa$-Borel set is any set of this collection. 

A function $f\colon \kappa^\kappa\rightarrow \kappa^\kappa$ is $\kappa$-Borel,
if for every open set $A\subseteq \kappa^\kappa$ the inverse image
$f^{-1}[A]$ is a $\kappa$-Borel subset of $X$. Let $E_1$ and $E_2$ be
equivalence relations on $\kappa^\kappa$. We say that $E_1$ is
\emph{$\kappa$-Borel reducible to $E_2$} if there is a $\kappa$-Borel function $f\colon
\kappa^\kappa\rightarrow \kappa^\kappa$ that satisfies $$(\eta,\xi)\in E_1\iff
(f(\eta),f(\xi))\in E_2.$$  We call $f$ a reduction of $E_1$ to
$E_2$ and we denote this by $E_1\redub E_2$. We will use this notation instead of ($\leq_B$), because we will deal with the equivalence relations $=_S^\beta$ (Definition \ref{clubrel}) and the notation could become heavy for the reader.
 In the case $f$ is continuous,
 we say that $E_1$ is continuously reducible to $E_2$ and
we denote it by $E_1\reduc E_2$.

A subset $X\s \kappa^\kappa$ is a $\Sigma_1^1(\kappa)$ set of $\kappa^\kappa$ if there is a closed set $Y\s \kappa^\kappa\times \kappa^\kappa$ such that the projection $\pr(Y):=\{x\in \kappa^\kappa\mid \exists y\in\kappa^\kappa,~(x,y)\in Y\}$ is equal to $X$. 
These definitions also extended to the product space $\kappa^\kappa\times\kappa^\kappa$.
An equivalence relation $E$ is $\Sigma_1^1$-complete if $E$ is a $\Sigma_1^1(\kappa)$ set and every $\Sigma_1^1(\kappa)$ equivalence relation $R$ is Borel reducible to $E$.

The generalized Cantor space is the subspace $2^\kappa$.
Since in this article we will only work with $\kappa$-Borel and $\Sigma_1^1(\kappa)$ sets, we will omit $\kappa$, and refer to them as Borel and $\Sigma_1^1$.

\begin{defn}\label{clubrel}
Given $S\subseteq \kappa$ and $\beta\leq\kappa$, we define the equivalence relation $=_S^\beta\ \subseteq\ \beta^\kappa\times \beta^\kappa$, 
as follows $$\eta\mathrel{=^\beta_S}\xi \iff \{\alpha<\kappa\mid \eta(\alpha)\neq\xi(\alpha)\}\cap S \text{ is non-stationary}.$$
\end{defn}
We will denote by $=_\mu^\beta$ the relation $=_S^\beta$ when $S=\{\alpha<\kappa\mid cf(\alpha)=\mu\}$. Let us denote by $CUB$ the club filter on $\kappa$ and $=^\beta_{CUB}$ the relation $=_S^\beta$ when $S=\kappa$.

\begin{defn}\label{struct}
Let $\mathcal{L}=\{Q_m\mid m\in\omega\}$ be a countable relational language.
Fix $\pi$ a bijection between $\kappa^{<\omega}$ and $\kappa$. For every $\eta\in \kappa^\kappa$ define the structure $\mathcal{A}_\eta$ with domain $\kappa$ as follows.
For every tuple $(a_1,a_2,\ldots , a_n)$ in $\kappa^n$ $$(a_1,a_2,\ldots , a_n)\in Q_m^{\mathcal{A}_\eta}\Leftrightarrow Q_m \text{ has arity } n \text{ and }\eta(\pi(m,a_1,a_2,\ldots,a_n))>0.$$
\end{defn}

\begin{defn}
Assume $T$ a first-order theory in a relational countable language, we define the isomorphism relation, $\cong_T~\subseteq \kappa^\kappa\times \kappa^\kappa$, as the relation $$\{(\eta,\xi)|(\mathcal{A}_\eta\models T, \mathcal{A}_\xi\models T, \mathcal{A}_\eta\cong \mathcal{A}_\xi)\text{ or } (\mathcal{A}_\eta\not\models T, \mathcal{A}_\xi\not\models T)\}$$
\end{defn}



\section{Ordered trees}\label{section_colorable_linear_order}

\subsection{Background}

In \cite{Sh}, Shelah used ordered tree to construct non-isomorphic models. That construction was focus on obtaining non-isomorphic models, reason why we have to modify the trees to adapt the construction to the generalized Cantor space and such that for all $f,g\in2^\kappa$, $f$ and $g$ are $=_\omega^\beta$-equivalent if and only if the constructed models are isomorphic. Let us start by reviewing the trees used by Shelah.

Let $\gamma$ be a countable ordinal, we will denote by $K_{tr}^\gamma$ the class of ordered trees with $\gamma+1$ levels.

\begin{defn}\label{Shtree}
Let $K_{tr}^\gamma$ be the class of models $(A,\prec, (P_n)_{n\leq \gamma},<, h)$, where:
\begin{enumerate}
\item there is a linear order $(I,<_I)$ such that $A\subseteq I^{\leq\gamma}$;
\item $A$ is closed under initial segment;
\item $\prec$ is the initial segment relation;
\item $h(\eta,\xi)$ is the maximal common initial segment of $\eta$ and $\xi$;
\item let $lg(\eta)$ be the length of $\eta$ (i.e. the domain of $\eta$) and $P_n=\{\eta\in A\mid lg(\eta)=n\}$ for $n\leq \gamma$;
\item for every $\eta\in A$ with $lg(\eta)<\gamma$, define $Suc_A(\eta)$ as $\{\xi\in A\mid \eta\prec \xi \wedge lg(\xi)=lg(\eta)+1\}$. If $\xi<\zeta$, then there is $\eta\in A$ such that $\xi,\zeta\in Suc_A(\eta)$; 
\item for every $\eta\in A\backslash P_\gamma$, $<\restriction Suc_A(\eta)$ is the induced linear order from $I$, i.e. $$\eta^\frown \langle x\rangle<\eta^\frown \langle y\rangle \Leftrightarrow x<_I y;$$
\item If $\eta$ and $\xi$ have no immediate predecessor  and $\{\zeta\in A\mid \zeta\prec\eta\}=\{\zeta\in A\mid \zeta\prec\xi\}$, then $\eta=\xi$.
\end{enumerate}
\end{defn}

To construct the models of unsuperstable theories, Shelah study the types of the ordered trees. To do this study, the notions of $\kappa$-representation and  $CUB$-invariant are crucial.

\begin{defn}[$\kappa$-representation]
Let $A$ be an arbitrary set of size at most $\kappa$. The sequence $\mathbb{A}=\langle A_\alpha\mid \alpha<\kappa  \rangle$ is a $\kappa$-representation of $A$, if $\langle A_\alpha\mid \alpha<\kappa  \rangle$ is an increasing continuous sequence of subsets of $A$, for all $\alpha<\kappa$, $|A_\alpha|<\kappa$, and $\bigcup_{\alpha<\kappa}A_\alpha=A$.
\end{defn}

\begin{defn}[$CUB$-invariant]
A function $\mathcal{H}$ is $CUB$-invariant if the following holds:
\begin{itemize}
\item The domain of $\mathcal{H}$ is the class  of $\kappa$-representations of the models of some model class $K$, where $K$ contains only models of size at most $\kappa$.
\item If $\mathbb{I}_1$ and $\mathbb{I}_2$ are $\kappa$-representations of $\mathcal{I}_1, \mathcal{I}_2\in K$, respectively, and $\mathcal{I}_1\cong \mathcal{I}_2$, then $\mathcal{H}(\mathbb{I}_1)=^2_{CUB}\mathcal{H}(\mathbb{I}_2)$.
\end{itemize}
\end{defn} 

Let us define for every $\mathcal{H}$ $CUB$-invariant and $A\in K_{tr}^\omega$, $\mathcal{H}(A)$ as the $=^2_{CUB}$-equivalence class of any $\mathbb{A}$, $\kappa$-representation, i.e. $[\mathcal{H}(\mathbb{A})]_{=^2_{CUB}}$.

We will use some properties of formulas and types.
For any $\mathcal{L}$-structure $M$ we denote by \textit{at} the set of atomic formulas of $\mathcal{L}$ and by \textit{bs} the set of basic formulas of $\mathcal{L}$ (atomic formulas and negation of atomic formulas). For all $\mathcal{L}$-structure $M$, $a\in M$, and $B\subseteq M$ we define $$tp_{bs}(a,B,M)=\{\varphi(x,b)\mid M\models \varphi(a,b),\varphi\in bs,b\in B\}.$$
In the same way $tp_{at}(a,B,M)$ is defined.

\begin{defn}
Let $\mathcal{A}$ be a model, $a\in \mathcal{A}$, $B,D\subseteq \mathcal{A}$.
We say that $tp_{bs}(a,B,\mathcal{A})$ (\textit{bs,bs})-splits over $D\subseteq \mathcal{A}$ if there are $b_1,b_2\in B$ such that $tp_{bs}(b_1,D,\mathcal{A})=tp_{bs}(b_2,D,\mathcal{A})$ but $tp_{bs}(a^\frown b_1,D,\mathcal{A})\neq tp_{bs}(a^\frown b_1,D,\mathcal{A})$.
\end{defn}

\begin{defn}
Let $|A|\leq \kappa$, for a $\kappa$-representation $\mathbb{A}$ of $A$. Define $Sp_{bs}(\mathbb{A})$ as $$Sp_{bs}(\mathbb{A})=\{\delta<\kappa\mid \delta\text{ a limit ordinal, }\exists a\in A\ [\forall\beta<\delta \ (tp_{bs}(a,A_\delta,A)\text{ (bs,bs)-splits over }A_\beta)]\}.$$
\end{defn}

\begin{remark}
The function $Sp_{bs}$ is $CUB$-invariant, this was stated in \cite{Sh} and proved in \cite{HT}. This is generally true under the assumption that for all $\gamma<\kappa$, $\gamma^\omega<\kappa$, which is one of our cardinal assumptions on $\kappa$ above.
\end{remark}

\begin{defn}
\begin{itemize}
\item Let $\mathcal{A}$ be a model of size at most $\kappa$. We say that $A$ is $(\kappa, bs,bs)$-nice if $Sp_{bs}(\mathbb{A})\ =^2_{CUB}\ \emptyset$.
\item $A\in K^\omega_{tr}$ of size at most $\kappa$, is locally $(\kappa, bs,bs)$-nice if for every $\eta\in A\backslash P^A_\omega$, $(Suc_A(\eta),<)$ is $(\kappa, bs,bs)$-nice, $Suc_A(\eta)$ is infinite, and there is $\xi\in P_\omega^A$ such that $\eta\prec\xi$.
\item $A\in K^\omega_{tr}$ is $(<\kappa, bs)$-stable if for every $B\subseteq A$ of size smaller than $\kappa$, $$\kappa>|\{tp_{bs}(a,B,A)\mid a\in A\}|.$$
\end{itemize}
\end{defn}

In \cite{Sh}, Shelah used $(<\kappa, bs)$-stable locally $(\kappa, bs,bs)$-nice ordered trees to construct the models of unsuperstable theories. In \cite{HT} Hyttinen and Tuuri give a very good example of a $(<\kappa, bs)$-stable $(\kappa, bs,bs)$-nice linear order, which is crucial for the construction of ordered trees.

\begin{defn}[Hyttinen-Tuuri, \cite{HT}]\label{HT-order}
Let $\mathcal{R}$ be the set of functions $f:\omega\rightarrow\kappa$ for which $\{n\in\omega\mid f(n)\neq 0\}$ is finite. If $f,g\in \mathcal{R}$, then $f<g$ if and only if $f(n)<g(n)$, where $n$ is the least number such that $f(n)\neq g(n)$.
\end{defn}

\begin{fact}[Hyttinen-Tuuri, \cite{HT}, Lemma 8.17]\label{HT-stable}
\begin{itemize}
\item The linear order $\mathcal{R}$ is $(<\kappa, bs)$-stable and $(\kappa, bs,bs)$-nice.
\item There is a $\kappa$-representation $\langle R_\alpha\mid\alpha<\kappa\rangle$ and a club $C\subseteq \kappa$ such that for all $\delta\in C$ and $\nu\in \mathcal{R}$ there is $\beta<\delta$ which satisfies the following:
$$\forall\sigma\in R_\delta [\sigma>\nu \Rightarrow \exists\sigma'\in R_\beta\ (\sigma\ge\sigma'\ge\nu)]$$
\end{itemize}
\end{fact}

\subsection{Colorable orders}

As it was mentioned in the previous subsection, the linear order plays a crucial roll when we construct the ordered trees and therefore the models. For our purpose to construct ordered trees from colored trees, we will need to choose the right linear order. The linear order that we will use are the colorable linear orders.

\begin{defn}
Let $I$ be a linear order of size $\kappa$. We say that $I$ is $\kappa$-colorable if there is a function $F: I\rightarrow\kappa$ such that for all $B\subseteq I$, $|B|<\kappa$, $b\in I\backslash B$, and $p=tp_{bs}(b,B,I)$ such that the following hold:
For all $\alpha\in \kappa$, $|\{a\in I\mid a\models p\ \&\  F(a)=\alpha\}|=\kappa$.
\end{defn}
 We say that $F$ is a $\kappa$-coloration of $I$, if $F$ witnesses that I is a $\kappa$-colorable linear order.

We will modify the order of Definition \ref{HT-order} to construct a $(<\kappa, bs)$-stable $(\kappa, bs,bs)$-nice  $\kappa$-colorable linear order.

\begin{defn}
Let $\mathbb{Q}$ be the linear order of the rational numbers.
Let $\kappa\times\mathbb{Q}$ be order by the lexicographic order, $I^0$ be the set of functions $f:\omega\rightarrow\kappa\times\mathbb{Q}$ such that $f(n)=(f_1(n),f_2(n))$, for which $\{n\in\omega\mid f_1(n)\neq 0\}$ is finite. If $f,g\in I^0$, then $f<g$ if and only if $f(n)<g(n)$, where $n$ is the least number such that $f(n)\neq g(n)$.
\end{defn}

\begin{lemma}\label{first_order_result}
There is a $\kappa$-representation $\langle I^0_\alpha\mid\alpha<\kappa\rangle$ such that for all limit $\delta<\kappa$ and $\nu\in I^0$ there is $\beta<\delta$ which satisfies the following:
$$\forall\sigma\in I^0_\delta [\sigma>\nu \Rightarrow \exists\sigma'\in I^0_\beta\ (\sigma\ge\sigma'\ge\nu)]$$
\end{lemma}

\begin{proof}
For all $\gamma<\kappa$, let us define $\langle I^0_\alpha\mid\alpha<\kappa\rangle$ by $$I^0_\gamma=\{\nu\in I^0\mid \nu_1(n)<\gamma\textit{ for all } n<\omega\}$$
it is clear that $\langle I^0_\alpha\mid\alpha<\kappa\rangle$ is a $\kappa$-representation. 

Suppose $\delta<\kappa$ is a limit and $\nu\in I^0$. If $\nu\in I^0_\delta$, then there is $\beta<\delta$ such that $\nu\in I^0_\beta$ and the result follows.

Let us take care of the case $\nu\notin I^0_\delta$. Let $\beta<\delta$ be the least ordinal such that for all $n<\omega$, $\nu_1(n)<\delta$ implies $\nu_1(n)<\beta$.

\begin{claim}
For all $\sigma\in I^0_\delta$. If $\sigma> \nu$, then there is $\sigma'\in I^0_\beta$ such that $\sigma\neq\sigma'$ and $\sigma>\sigma'>\delta$.
\end{claim}
\begin{proof}
Let us suppose $\sigma\in I^0_\delta$ is such that $\sigma\ge \nu$. By the definition of $I^0$, there is $n<\omega$ such that $\sigma(n)>\nu(n)$ and $n$ is the minimum number such that $\sigma(n)\neq \nu(n)$. Since $\sigma\in I^0_\delta$, for all $m\leq n$, $\nu_1(m)\leq \sigma_1(m)<\delta$. Thus for all $m\leq n$, $\nu_1(m)<\beta$.

Let us divide the proof in two cases, $\sigma_1(n)=\nu_1(n)$ and $\sigma_1(n)>\nu_1(n)$.

{\bf Case 1.} $\sigma_1(n)=\nu_1(n)$.

By the density of $\mathbb{Q}$ there is $r$ such that $\sigma_2(n)>r>\nu_2(n)$. Let us define $\sigma'$ by:
$$\sigma'(m)=\begin{cases} \nu(m) &\mbox{if } m<n\\
(\nu_1(n), r) & \mbox{if } m=n\\
0 & \mbox{in other case. } \end{cases}$$  
Clearly $\sigma>\sigma'>\nu$. Since $\nu_1(m)<\beta$ for all $m\leq n$, $\sigma'\in I^0_\beta$.

{\bf Case 2.} $\sigma_1(n)>\nu_1(n)$.

Let us define $\sigma'$ by:
$$\sigma'(m)=\begin{cases} \nu(m) &\mbox{if } m<n\\
(\nu_1(n), \nu_2(n)+1) & \mbox{if } m=n\\
0 & \mbox{in other case. } \end{cases}$$  
Clearly $\sigma>\sigma'>\nu$. Since $\nu_1(m)<\beta$ for all $m\leq n$, $\sigma'\in I^0_\beta$.
\end{proof}
\end{proof}

\begin{cor}\label{stricti_minor}
There is a $\kappa$-representation $\langle I^0_\alpha\mid\alpha<\kappa\rangle$ such that for all limit $\delta<\kappa$ and $\nu\in I^0$, if $\nu\notin I^0_\delta$ there is $\beta<\delta$ which satisfies the following:
$$\forall\sigma\in I^0_\delta [\sigma>\nu \Rightarrow \exists\sigma'\in I^0_\beta\ (\sigma>\sigma'>\nu)]$$
\end{cor}

Now let us used the order $I^0$ to construct a $(<\kappa, bs)$-stable, $(\kappa, bs,bs)$-nice, and $\kappa$-colorable linear order.
Let us construct the linear orders $\langle I^i\mid i<\kappa\rangle$ by induction, such that for all $i<j$, $I^i\subseteq I^j$.
Suppose $i<\kappa$ is such that $I^i$ has been defined. For all $\nu\in I^i$ let $\nu^{i+1}$ be such that 
\begin{equation}\label{generated_element}
\nu^{i+1}\models tp_{bs}(\nu, I^i\backslash \{\nu\}, I^i)\cup\{\nu>x\}.
\end{equation}
Notice that $\nu^{i+1}$ is a copy of $\nu$ that is smaller than $\nu$. Let $I^{i+1}=I^i\cup\{\nu^{i+1}\mid \nu\in I^i\}$.

Suppose $i<\kappa$ is a limit ordinal such that for all $j<i$, $I^j$ has been defined, we define $I^i$ by $I^i=\bigcup_{j<i}I^j$.
 
For all $i<\kappa$, let us define the $\kappa$-representation $\langle I^i_\alpha\mid\alpha<\kappa\rangle$ by induction as follows:

Suppose $i<\kappa$ is such that $\langle I^i_\alpha\mid\alpha<\kappa\rangle$ has been defined. For all $\alpha<\kappa$, $$I^{i+1}_\alpha=I^i_\alpha\cup \{\nu^{i+1}\mid \nu\in I^i_\alpha\}.$$
Suppose $i<\kappa$ is a limit ordinal such that for all $j<i$, $\langle I^j_\alpha\mid\alpha<\kappa\rangle$ has been defined, we define $\langle I^i_\alpha\mid\alpha<\kappa\rangle$ by $$I^i_\alpha=\bigcup_{j<i} I^j_\alpha .$$

Finally, let us define $I$ as $$I=\bigcup_{j<\kappa} I^j$$ and the $\kappa$-representation $\langle I_\alpha\mid\alpha<\kappa\rangle$ as $$I_\alpha=\bigcup_{\alpha<\kappa} I^\alpha_\alpha .$$

Before we are able to prove the main result of this section, we will need to develop the theory of $I$.

\begin{defn}[Generator]\label{Generator}
For all $\nu\in I$ let us denote by $o(\nu)$ the least ordinal $\alpha<\kappa$ such that $\nu\in I^\alpha$.
Let us denote the generator of $\nu$ by $Gen(\nu)$ and define it by induction as follows:
\begin{itemize}
\item $Gen^i(\nu)=\emptyset$, for all $i<o(\nu)$;
\item $Gen^i(\nu)=\{\nu\}$, for $i=o(\nu)$;
\item for all $i\ge o(\nu)$, $$Gen^{i+1}(\nu)=Gen^i(\nu)\cup\{\sigma\in I^{i+1}\mid \exists \tau\in Gen^i(\nu)\ [\tau^{i+1}=\sigma]\};$$
\item for all $i<\kappa$ limit, $$Gen^i(\nu)=\bigcup_{j<i}Gen^j(\nu).$$
\end{itemize} 
Finally, let $$Gen(\nu)=\bigcup_{i<\kappa}Gen^i(\nu).$$
\end{defn}

Notice that $o(\nu)$ is a successor ordinal for all $\nu$. For clarity purposes let us fix the following notation.

{\bf Notation.} For any sequence of length $\alpha$,  $\{\sigma_j\}_{j<\alpha}$ of elements of $I^i$, we will denote by $(\sigma_j)^{i+1}$ the element generated by $\sigma_j$ in $I^{i+1}$, i.e. $\sigma_j^{i+1}$  (see (\ref{generated_element}) above).

\begin{fact}
Suppose $\nu\in I$. For all $\sigma\in Gen(\nu)$, $\sigma\neq\nu$, there is $n<\omega$ and a sequence $\{\sigma_i\}_{i\leq n}$ such that the following holds:
\begin{itemize}
\item $\sigma_0=\nu$;
\item for all $j<n$, $$\sigma_{j+1}=(\sigma_j)^{o(\sigma_{j+1})};$$
\item $\sigma=\sigma_n=(\sigma_{n-1})^{o(\sigma)}$
\end{itemize}
\end{fact}

\begin{proof}
Let $\sigma\neq\nu$ be such that $\sigma\in Gen(\nu)$. From Definition \ref{Generator}, we know that there is $i<\kappa$ such that $\sigma\in Gen^{i+1}(\nu)$. Thus, there is $\tau\in I^i$ and $\tau^{i+1}=\sigma$. We conclude that there is a sequence of length $\alpha$, $\{\sigma_i\}_{i\leq \alpha}$ such that the following holds:
\begin{itemize}
\item $\sigma_0=\nu$;
\item for all $j<\alpha$, $$\sigma_{j+1}=(\sigma_j)^{o(\sigma_{j+1})};$$
\item there is $\alpha$, such that $\sigma=\sigma_\alpha=(\sigma_{\alpha-1})^{o(\sigma)}$.
\end{itemize}
On the other hand, since there are no infinite decreasing sequence of ordinals, $\alpha$ is finite.
\end{proof}

For every $\nu\in I$, $\sigma\in Gen(\nu)$, and $\sigma\neq\nu$, we call the sequence $\{\sigma_i\}_{i\leq n}$ of the previous fact, \textit{the road from $\nu$ to $\sigma$}. It is clear that for all $\nu\in I\backslash I^0$, there is $\nu'\in I^0$ such that $\nu\in Gen(\nu')$. Notice that for all $\nu\in I$, if $\sigma\in Gen(\nu)$, then $\nu$ and $\sigma$ have the same type of basic formulas over $I^{o(\nu)}\backslash \{\nu\}$. Even more, if $\{\sigma_i\}_{i\leq n}$ is the road from $\nu$ to $\sigma$, then for all $i< n$, $\sigma_i$ and $\sigma$ have the same type of basic formulas over $I^\gamma\backslash\{\sigma_i\}$, where $o(\sigma_{i+1})=\gamma+1$.

\begin{fact}\label{element_of_delta}
Let $i,\delta, \nu$ be such that  $\nu\in I^i_\delta$. Then for all $\sigma\in Gen(\nu)$, $\sigma\in I^{o(\sigma)}_\delta$. In particular for all $j<\kappa$ $$\sigma\notin
I^j_\delta\Rightarrow \sigma\notin I^j.$$ 
\end{fact}
\begin{proof}
It follows from the construction of $I^{o(\sigma)}$ and the $\kappa$-representation $\langle I^{o(\sigma)}_\alpha\mid\alpha<\kappa\rangle$.
\end{proof}

\begin{fact}\label{Same_type_in_Gen}
For all $\nu,\sigma\in I$, $\sigma\in Gen(\nu)$, if $\sigma'\in I$ is such that $\nu\ge\sigma'\ge\sigma$, then $\sigma'\in Gen(\nu)$.
\end{fact}
\begin{proof}
If $\nu=\sigma$, the result follows. Thus we only need to prove the case $\nu\neq\sigma$. Let us suppose towards contradiction that $\sigma'\notin Gen(\nu)$.

{\bf Case $o(\nu)=o(\sigma')$.} Since $\nu$ and $\sigma$ have the same type of basic formulas over $I^{o(\nu)}\backslash \{\nu\}$, $\nu$ and $\sigma$ have the same type of basic formulas over $I^{o(\sigma')}\backslash \{\nu\}$.  Since $\nu\ge\sigma'\ge\sigma$, $\nu=\sigma'$ a contradiction.

{\bf Case $o(\sigma')<o(\nu)$.} Since $\nu\ge\sigma'$, there is $\nu'\neq\sigma'$ such that $\nu'>\nu$, $o(\nu')=o(\sigma')$ and $\nu\in Gen(\nu')$. Thus $\nu'$, $\sigma'$, and $\sigma$ satisfy $\nu'\ge\sigma'\ge\sigma$, $o(\nu')=o(\sigma')$, and $\sigma\in Gen(\nu')$. The result follows from the previous case.

{\bf Case $o(\nu)<o(\sigma')$.} There is $\sigma^0\in I$ such that $\sigma^0>\sigma'$, $o(\sigma^0)=o(\nu)$ and $\sigma'\in Gen(\sigma^0)$. If $\nu\ge\sigma^0\ge\sigma$, then the result follows from the previous cases. Therefore, we are only missing the case $\sigma^0\ge\nu\ge\sigma'\ge\sigma$. Since $\sigma^0$ and $\sigma'$ have the same type of basic formulas of basic formulas over $I^{o(\sigma^0)}\backslash\{\sigma^0\}$, $\sigma^0=\nu$ and $\sigma'\in  Gen(\nu)$ a contradiction.
\end{proof}

From the previous fact we can conclude that for all $\nu,\sigma\in I$ such that $\sigma\in Gen(\nu)$, $\nu$ and $\sigma$ have the same type of basic formulas over $I\backslash Gen(\nu)$.

\begin{lemma}\label{nice_for_i}
For all $i<\kappa$, $\delta<\kappa$ a limit ordinal, and $\nu\in I^i$, there is $\beta<\delta$ that satisfies the following:
\begin{equation}
\forall\sigma\in I^i_\delta\ [\sigma>\nu \Rightarrow \exists\sigma'\in I^i_\beta\ (\sigma\ge\sigma'\ge\nu)]
\end{equation}
\end{lemma}

\begin{proof}
Notice that if $\nu\in I^{i}_\delta$, then there is $\theta<\delta$ such that $\nu\in I^{i}_\theta$ and the result follows for $\beta=\theta$. So we only have to prove the lemma when $\nu\in I^i\backslash  I^{i}_\delta$. Let us prove something stronger:

\textit{For all $i<\kappa$, $\delta<\kappa$ a limit ordinal, and $\nu\in I^i\backslash I^{i}_\delta$, there is $\beta<\delta$ that satisfies the following:}
\begin{equation}\label{property_HT2}
\forall\sigma\in I^i_\delta\ [\sigma>\nu \Rightarrow \exists\sigma'\in I^0_\beta\ (\sigma>\sigma'>\nu)]
\end{equation}

We will proceed by induction over $i$. The case $i=0$ is precisely Corollary \ref{stricti_minor}. Let us suppose $i<\kappa$ is such that for all limit ordinal $\delta<\kappa$ and $\nu\in I^i\backslash I^{i}_\delta$, there is $\beta<\delta$ that satisfies (\ref{property_HT2}). Let $\delta<\kappa$ be a limit ordinal and $\nu\in I^{i+1}\backslash I^{i+1}_\delta$.
We have two cases, $\nu\in I^i$ and $\nu\in I^{i+1}\backslash I^i$.

{\bf Case $\nu\in I^i$.} By the induction hypothesis, we know that there is $\beta<\delta$ such that (\ref{property_HT2}) holds.
Let us prove that this $\beta$ is the one we are looking for. Let $\sigma\in I^{i+1}_\delta$ be such that $\sigma>\nu$.
The subcase $\sigma\in I^i_\delta$ follows from the way $\beta$ was chosen. 

{\bf Subcase $\sigma\in I^{i+1}_\delta\backslash I^i_\delta$.} By the construction of $I^{i+1}$, there is $\sigma_0\in I^i_\delta$ such that $\sigma=(\sigma_0)^{i+1}$ (so $\sigma_0>\sigma$). Thus $\sigma_0>\sigma>\nu$, and by the way $\beta$ was chosen, there is $\sigma'\in I^0_\beta$ such that $\sigma_0> \sigma'> \nu$. Since $\sigma_0$ and $\sigma$ have the same type of basic formulas over $I^i\backslash \{\sigma_0\}$, $\sigma> \sigma'> \nu$ as we wanted.

{\bf Case $\nu\in I^{i+1}\backslash I^i$.} By the  construction of $I^{i+1}$, there is $\nu_0\in I^i$ such that $(\nu_0)^{i+1}=\nu$. Since $\nu\in Gen(\nu_0)$ and $\nu\in I^{i+1}\backslash I^{i+1}_\delta$, by Fact \ref{element_of_delta} $\nu_0\in I^i\backslash I^i_\delta$. Thus, by the previous case, there is $\beta<\delta$ such that  for all $\sigma\in I^{i+1}_\delta$:
$$\sigma>\nu_0 \Rightarrow \exists\sigma'\in I^0_\beta\ (\sigma>\sigma'>\nu_0).$$
Let us show that this $\beta$ is as wanted.
\begin{claim}
If $\sigma\in I^{i+1}_\delta$ is such that $\sigma>\nu$, then $\sigma>\nu_0$.
\end{claim}
\begin{proof}
Let us suppose, towards contradiction, that there is $\sigma\in I^{i+1}_\delta$ such that $\nu_0>\sigma>\nu$. Since $\nu_0$ and $\nu$ have the same type of basic formulas over $I^i\backslash \{\nu_0\}$, $\sigma\in I^{i+1}_\delta\backslash I^i$. Therefore, there is $\sigma_0\in I^i$ such that $(\sigma_0)^{i+1}=\sigma$. Since $\sigma\in Gen(\sigma_0)$ and $\sigma\in I^{i+1}_\delta$, $\sigma_0\in I^i_\delta$. We conclude that $\sigma_0\neq \nu_0$. Finally, $\sigma_0$ and $\sigma$ have the same type of basic formulas over $I^i\backslash \{\sigma_0\}$, which implies $\nu_0>\sigma_0>\sigma>\nu$. This contradicts the fact that $\nu_0$ and $\nu$ have the same type of basic formulas over $I^i\backslash \{\nu_0\}$.
\end{proof}
From the previous claim, we know that for all $\sigma\in I^{i+1}_\delta$, $\sigma>\nu$ implies $\sigma>\nu_0$. By the way $\beta$ was chosen we conclude that for all $\sigma\in I^{i+1}_\delta$, $\sigma>\nu$ implies the existence of $\sigma'\in I^0_\beta$ such that $\sigma>\sigma'>\nu_0>\nu$, as we wanted.

Let us proceed with the limit case. Suppose $i<\kappa$ is a limit ordinal such that for all $j<i$, for all limit ordinal $\delta<\kappa$, and $\nu\in I^j\backslash I^j_\delta$,  there is $\beta<\delta$ such that (\ref{property_HT2}) holds for $j$.
Let $\delta<\kappa$ be a limit ordinal and $\nu\in I^i\backslash I^i_\delta$. Since $i$ is a limit, $o(\nu)<i$, by the induction hypothesis, there is $\beta$ such that (\ref{property_HT2}) holds for $o(\nu)$. 
\begin{claim}
$\beta$ is as wanted.
\end{claim}
\begin{proof}
Let $\sigma\in I^i_\delta$ be such that $\sigma>\nu$.

{\bf Case $\sigma\in I^{o(\nu)}_\delta$.} This case follows from the way $\beta$ was chosen.

{\bf Case $\sigma\in I^i_\delta\backslash I^{o(\nu)}_\delta$.} There is $\sigma_0\in I^{o(\nu)}_\delta$ such that $\sigma\in Gen(\sigma_0)$, with road to $\sigma$ equal to $\{\sigma_i\}_{i\leq n}$ such that $\sigma_1\notin I^{o(\nu)}$. Therefore $\sigma_0$ and $\sigma$ have the same type of basic formulas over $I^\gamma\backslash \{\sigma_0\}$, where $o(\sigma_1)=\gamma+1$. In particular $\sigma_0$ and $\sigma$ have the same type of basic formulas over $I^{o(\nu)}\backslash \{\sigma_0\}$. By the way $\beta$ was chosen, there is $\sigma'\in I^0_\beta\subseteq I^{o(\nu)}_\beta$ such that $\sigma_0>\sigma'>\nu$. Since $\sigma_0$ and $\sigma$ have the same type of basic formulas over $I^{o(\nu)}\backslash \{\sigma_0\}$,   $\sigma>\sigma'>\nu$ as wanted.
\end{proof}
\end{proof}

As it can be seen in the proof of the previous lemma, the witness $\sigma'$ can be chosen in $I^0_\beta$ when $\nu\notin I^i_\delta$.

\begin{cor}\label{strict_minor_in0}
For all $i<\kappa$, $\delta<\kappa$ a limit ordinal, and $\nu\in I^i$, if $\nu\notin I^i_\delta$ there is $\beta<\delta$ which satisfies the following:
$$\forall\sigma\in I^i_\delta [\sigma>\nu \Rightarrow \exists\sigma'\in I^0_\beta\ (\sigma>\sigma'>\nu)]$$
\end{cor}

\begin{lemma}\label{nice_big_I}
For all $\delta<\kappa$ limit, and $\nu\in I$, there is $\beta<\delta$ that satisfies the following:
\begin{equation}\label{property_HT3}
\forall\sigma\in I_\delta\ [\sigma>\nu \Rightarrow \exists\sigma'\in I_\beta\ (\sigma\ge\sigma'\ge\nu)]
\end{equation}
\end{lemma}

\begin{proof}
Let $\delta<\kappa$ be a limit ordinal, and $\nu\in I$.
We have three different cases: $\nu\in I_\delta$, $\nu\in I^{o(\nu)}_\delta \backslash I_\delta$, and $\nu\notin I^{o(\nu)}_\delta$.

{\bf Case $\nu\in I_\delta$.} Since $\delta$ is a limit, $o(\nu)<\delta$ and there is $\theta<\delta$ such that $\nu\in I^{o(\nu)}_\theta$. Let $\beta=max\{o(\nu),\theta\}$, it is clear that $\beta$ is as wanted.

{\bf Case $\nu\in I^{o(\nu)}_\delta \backslash I_\delta$.} Recall $I_\delta=I^\delta_\delta$, clearly $\delta<o(\nu)$. There is $\nu_0\in I_\delta$, such that $\nu\in Gen(\nu_0)$, with road to $\nu$ equal to $\{\nu_i\}_{i\leq n}$, and $\nu_1\notin I^\delta$. Since $\nu_0\in I^\delta_\delta$ and $\delta$ is a limit, $o(\nu_0)<\delta$ and there is $\theta<\delta$ such that $\nu_0\in I^{o(\nu_0)}_\theta$. Let $\beta=max\{o(\nu_0),\theta\}$.

\begin{claim}
$\beta$ is as wanted.
\end{claim}
\begin{proof}
Let $\sigma\in I^\delta_\delta$ be such that $\sigma>\nu$. Since $\nu_1\notin I^\delta$, $o(\nu_1)=\gamma+1>\delta$, and $\nu_0$ and $\nu$ have the same type of basic formulas over $I^\gamma\backslash\{\nu_0\}$. In particular $\nu_0$ and $\nu$ have the same type of basic formulas over $I^\delta\backslash \{\nu_0\}$, so $\sigma>\nu_0>\nu$. Since $\nu_0\in I^\beta_\beta$, $\sigma'=\nu_0$ is as wanted.
\end{proof}

{\bf Case $\nu\notin I^{o(\nu)}_\delta$.}  Let $\theta=max\{o(\nu), \delta\}$, thus $\nu\in I^\theta$ (notice that we are talking about the order $I^\theta$ and not the element $I_\theta$ of the $\kappa$-representation $\langle I_\alpha\mid \alpha<\kappa\rangle$) and by Corollary \ref{strict_minor_in0} there is $\beta<\delta$ which satisfies the following:
$$\forall\sigma\in I^\theta_\delta [\sigma>\nu \Rightarrow \exists\sigma'\in I^0_\beta\ (\sigma>\sigma'>\nu)].$$

\begin{claim}
$\beta$ is as wanted.
\end{claim}
\begin{proof}
Let $\sigma\in I^\delta_\delta$ be such that $\sigma>\nu$. Since $\delta\leq\theta$, $\sigma\in I^\theta_\delta$. Therefore, there is $\sigma'\in I^0_\beta$ such that $\sigma>\sigma'>\nu$. The proof follows from $I^0_\beta\subseteq I^\beta_\beta=I_\beta$.
\end{proof}
\end{proof}

\begin{fact}[Hyttinen-Tuuri, \cite{HT} Lemma 8.12]
Let $A$ be a linear order of size $\kappa$ and $\langle A_\alpha\mid\alpha<\kappa\rangle$ a $\kappa$-representation. Then the following are equivalent:
\begin{enumerate}
\item $A$ is $(\kappa, bs,bs)$-nice.
\item There is a club $C\subseteq \kappa$, such that for all limit $\delta\in C$, for all $x\in A$ there is $\beta<\delta$ such that one of the following holds:
\begin{itemize}
\item $\forall\sigma\in A_\delta [\sigma\ge x \Rightarrow \exists\sigma'\in A_\beta\ (\sigma\ge \sigma'\ge x)]$
\item $\forall\sigma\in A_\delta [\sigma\leq x \Rightarrow \exists\sigma'\in A_\beta\ (\sigma\leq \sigma'\leq x)]$
\end{itemize}
\end{enumerate}
\end{fact}

From Lemma \ref{nice_big_I} it follows the next corollary. 

\begin{cor}\label{I_nice}
$I$ is $(\kappa, bs,bs)$-nice.
\end{cor}

Notice that if $\kappa$ is inaccessible, $I$ is $(<\kappa, bs)$-stable. This can be generalize to $\kappa$ successors.

\begin{lemma}\label{stable_I0_first_lemma}
Suppose $\kappa=\lambda^+$.
$I^0$ is $(<\kappa, bs)$-stable.
\end{lemma}

\begin{proof}
Recall the linear order $\mathcal{R}$ from Definition \ref{HT-order}. From the general assumption on $\kappa$, we know that $\lambda^\omega=\lambda$.

For all $A\subseteq I^0$ define $Pr(A)$ as the set $\{f_1\mid f\in A\}$. Let $A\subseteq I^0$ be such that $|A|<\kappa$. 
Since $|\mathbb{Q}|=\omega$, $|\{tp_{bs}(a,A,I^0)\mid a\in I^0\}|\leq |\{tp_{bs}(a,Pr(A),\mathcal{R})\mid a\in \mathcal{R}\}\times 2^\omega|$. 
By Fact \ref{HT-stable} and since $\lambda^\omega=\lambda$, $|\{tp_{bs}(a,A,I^0)\mid a\in I\}|< \kappa$.
\end{proof}

\begin{lemma}\label{I_stable}
Suppose $\kappa=\lambda^+$.
$I$ is $(<\kappa, bs)$-stable.
\end{lemma}

\begin{proof}
Let us fix $A\subset I$ such that $|A|<\kappa$. From Fact \ref{Same_type_in_Gen}, for all $a\in I$ and $\nu\in I^0$ such that $a\in Gen(\nu)$ the following holds:
$$b\models tp_{bs}(a,A,I) \Leftrightarrow b\models tp_{bs}(\nu,A \backslash Gen(\nu),I)\cup tp_{bs}(a,A\cap Gen(\nu), Gen(\nu)).$$
Thus for all $a\in I$ and $\nu\in I^0$ with $a\in Gen(\nu)$, the type of $a$ is determine by $tp_{bs}(\nu,A \backslash Gen(\nu),I)$ and $tp_{bs}(a,A \cap Gen(\nu), Gen(\nu))$.
Let $A'\subseteq I^0$ be such that the following hold:
\begin{itemize}
\item for all $x\in A$ there is $y\in A'$, $x\in Gen(y)$;
\item for all $y\in A'$ there is $x\in A$, $x\in Gen(y)$.
\end{itemize}

Clearly $|A'|\leq |A|$, and by Fact \ref{Same_type_in_Gen}, for all $\nu\in I^0$, $tp_{bs}(\nu,A\backslash Gen(\nu),I)$ is determine by $tp_{bs}(\nu,A'\backslash \{\nu\},I^0)$. So for all $a\in I$ and $\nu\in I^0$ with $a\in Gen(\nu)$, $tp_{bs}(a,A,I)$ is determine by $tp_{bs}(\nu,A'\backslash \{\nu\},I^0)$ and $tp_{bs}(a,A\cap Gen(\nu), Gen(\nu))$. Therefore $|\{tp_{bs}(a,A,I)\mid a\in I\}|$ is bounded by 
$$|\{tp_{bs}(\nu,A',I^0)\mid \nu\in I^0\}|\ \times\  Sup(\{B_\nu \mid \nu \in I^0\})$$ where 
$$B_\nu=|\{tp_{bs}(a,A \cap Gen(\nu), Gen(\nu))\mid a\in Gen(\nu)\}|.$$

\begin{claim}
For all $\nu\in I^0$, $Gen(\nu)$ with the induced order is $(<\kappa, bs)$-stable.
\end{claim}
\begin{proof}
Let us  fix $\nu\in I^0$, $\sigma\in Gen(\nu)\backslash \{\nu\}$, and let $\{\nu_i\}_{i\leq n}$ be the road from $\nu$ to $\sigma$. Let us define $f_\sigma: \omega\rightarrow \kappa$ by  
$$f_\sigma(i)=\begin{cases} o(\nu_i) &\mbox{if } i< n\\
o(\sigma) & \mbox{if }i=n\\
0 & \mbox{in other case.	} \end{cases}$$
Notice that for all $\sigma,\sigma'\in Gen(\nu)$, $f_\sigma$ and $f_{\sigma'}$ are equal if and only if the road from $\nu$ to $\sigma$ is the same road from $\nu$ to $\sigma'$. Thus $f_\sigma=f_{\sigma'}$ if and only if $\sigma=\sigma'$. Since the road from $\nu$ to $\sigma$ is finite, $\{i<\omega\mid f_\sigma(i)\neq 0\}$ is finite. By the construction of $I$, for all $\sigma,\sigma'\in Gen(\nu)$, such that $\sigma,\sigma'\neq \nu$, $\sigma>\sigma'$ if and only if $f_\sigma(i)>f_{\sigma'}(i)$ where $i$ is the least number such that $f_\sigma(i)\neq f_{\sigma'}(i)$. Notice that $\nu$ is the maximum of $Gen(\nu)$. Let us define $f_\nu$ as 
$$f_\nu(i)=\begin{cases} 1 &\mbox{if } i=0\\
0 & \mbox{in other case.	} \end{cases}$$
so $f_\nu(0)>f_\sigma(0)$ for all $\sigma\in Gen(\nu)\backslash\{\nu\}$.
Therefore, $Gen(\nu)$ is isomorphic to a sub-order of $\mathcal{R}$ and by Fact \ref{HT-stable} $Gen(\nu)$ with the induced order is $(<\kappa, bs)$-stable.
\end{proof}
From the previous claim, we conclude that for all $\nu\in I^0$, $B_\nu<\kappa$. Since $\kappa=\lambda^+$, $Sup(\{B_\nu \mid \nu \in I^0\})\leq  \lambda$. From Lemma \ref{stable_I0_first_lemma} we know that $|\{tp_{bs}(\nu,A',I^0)\mid \nu\in I^0\}|<\kappa$, so $|\{tp_{bs}(\nu,A',I^0)\mid \nu\in I^0\}|\leq\lambda$. We conclude $|\{tp_{bs}(a,A,I)\mid a\in I\}|<\kappa$.
\end{proof}

\begin{thm}\label{A_desire_order}
There is a $(<\kappa, bs)$-stable $(\kappa, bs,bs)$-nice $\kappa$-colorable linear order.
\end{thm}

\begin{proof}
From Corollary \ref{I_nice} and Lemma \ref{I_stable}, we only need to show that $I$ is $\kappa$-colorable. 
For all $\nu\in I$ let us define $Succ_I(\nu)$ as follows:
$$Succ_I(\nu)=\{\sigma\in I\mid \sigma=\nu^{o(\sigma)}\}.$$ We use the same notation of ordered trees because $I$ can be seen as an ordered tree. Notice that for all $\nu\in I$, $|Succ_I(\nu)|=\kappa$ and either $o(\nu)=0$, or there is a unique $\nu'\in I$ such that $\nu=(\nu')^{o(\nu)}$ (i.e. $\nu\in Succ_I(\nu')$).

Let us fix $G:\kappa\rightarrow\kappa\times\kappa$ a bijection, and $G_1$, $G_2$ be the functions such that $G(\alpha)=(G_1(\alpha),G_2(\alpha))$.  For all $\nu\in I$ let us fix a bijection $g_\nu:Succ_I(\nu)\rightarrow \kappa$.
Let us define $F:I\rightarrow\kappa$ by 
$$F(\nu)=\begin{cases} 0 &\mbox{if } o(\nu)=0\\
G_1(g_{\nu'}(\nu)) & \mbox{where }(\nu')^{o(\nu)}=\nu. \end{cases}$$

\begin{claim}
$F$ is a $\kappa$-coloration of $I$.
\end{claim}
\begin{proof}
Let $B\subseteq I$, $|B|<\kappa$, $b\in I\backslash B$, and $p= tp_{bs}(b,B,I)$. Since $|B|<\kappa$, there is $\gamma<\kappa$ such that $B\subset I^{\gamma}$. Let $\theta=\max\{o(b), \gamma\}$, so for all $\nu\in \{a\in Succ_I(b)\mid o(a)>\theta\}$, $b$ and $\nu$ have the same type of basic formulas over $I^\theta\backslash \{b\}$. 
In particular for all $\nu\in \{a\in Succ_I(b)\mid o(a)>\theta\}$, $\nu\models p$. By the way $F$ was define, we conclude that for any $\alpha<\kappa$, $|\{a\in Succ_I(b)\mid o(a)>\theta\ \&\ F(a)=\alpha\}|=\kappa$. Which implies that for any $\alpha<\kappa$, $|\{a\in Succ_I(b)\mid a\models p\ \&\ F(a)=\alpha\}|=\kappa$
\end{proof}
\end{proof}

\section{Ordered Colored Trees}\label{section_ordered_colored_trees}

\subsection{Colored trees}

We will use the $\kappa$-colorable linear order $I$ to construct trees with $\omega+1$ levels, $A^f(I)$, for every $f\in \kappa^\kappa$ with the property $A^f(I) \cong A^g(I)$ if and only if $f\ =^\kappa_{\omega}\ g$. These tress will be a mix of colored tree and ordered trees. 
For clarity and to avoid misunderstandings, in this section we will denote trees by $(T,\prec)$. Later on we will see that $\prec$ is the initial segment relation of the trees that we construct .
The coloured trees that we will use in this section, are essentially the same trees used by Hyttinen and Kulikov in \cite{HK} and by Hyttinen and Moreno in \cite{HM}.

Let $t$ be a tree, for every $x\in t$ we denote by $ht(x)$ the height of $x$, the order type of $\{y\in t | y\prec x\}$. Define $(t)_\alpha=\{x\in t|ht(x)=\alpha\}$ and $(t)_{<\alpha}=\cup_{\beta<\alpha}(t)_\beta$, denote by $x\restriction \alpha$ the unique $y\in t$ such that $y\in (t)_\alpha$ and $y\prec x$. If $x,y\in t$ and $\{z\in t|z\prec x\}=\{z\in t|z\prec y\}$, then we say that $x$ and $y$ are $\sim$-related, $x\sim y$, and we denote by $[x]$ the equivalence class of $x$ for $\sim$.\\
An $\alpha, \beta$-tree is a tree $t$ with the following properties:
\begin{itemize}
\item $|[x]|<\alpha$ for every $x\in t$.
\item All the branches have order type less than $\beta$ in $t$.
\item $t$ has a unique root.
\item If $x,y\in t$, $x$ and $y$ have no immediate predecessors and $x\sim y$, then $x=y$.
\end{itemize}
\begin{defn}\label{D.2.1}
Let $\lambda$ be a cardinal smaller than $\kappa$, and $\beta$ an ordinal smaller or equal to $\kappa$. A coloured tree is a pair $(t,c)$, where $t$ is a $\kappa^+$, $(\lambda+2)$-tree and $c$ is a map $c:t_\lambda\rightarrow \beta$ (the color function). 
\end{defn}

Two coloured trees $(t,c)$ and $(t',c')$ are isomorphic, if there is a trees isomorphism $f:t\rightarrow t'$ such that for every $x\in t_\lambda$, $c(x)=c'(f(x))$.

We will only consider trees in which every element with height less than $\lambda$, has infinitely many immediate successors, every maximal branch has order type $\lambda+1$. Notice that the intersection of two distinct branches has order type less than $\lambda$. 
We can see every coloured tree as a downward closed subset of $\kappa^{\leq \lambda}$.
In this section all the coloured trees have $\lambda=\omega$.

An ordered coloured tree is a tree $T\in K_{tr}^\omega$ with a color function $c:t_\omega\rightarrow \beta$.

We will follow the construction used \cite{HK} and \cite{HM}. 

Let us start from coloured trees which are subsets of $(\omega\times\kappa^4)^{\leq\omega}$, let us make some preparation before the actual construction.
Order the set $\omega\times \kappa\times \kappa\times \kappa\times \kappa$ lexicographically, $(\alpha_1,\alpha_2,\alpha_3,\alpha_4,\alpha_5)>(\theta_1,\theta_2,\theta_3,\theta_4,\theta_5)$ if for some $1\leq k \leq 5$, $\alpha_k>\theta_k$ and for every $i<k$, $\alpha_i=\theta_i$. Order the set $(\omega\times \kappa\times \kappa\times \kappa\times \kappa)^{\leq \omega}$ as a tree by initial segments.

Define the tree $(R_f,r_f)$ as, $R_f$ the set of all strictly increasing functions from some $n\leq \omega$ to $\kappa$ and $r_f$ is the color function such that for each $\eta$ with domain $\omega$, $r_f(\eta)=f(sup(rng(\eta)))$.

For every pair of ordinals $\alpha$ and $\theta$, $\alpha<\theta<\kappa$ and $i<\omega$ define $$R(\alpha,\theta,i)=\bigcup_{i< j\leq \omega}\{\eta:[i,j)\rightarrow[\alpha,\theta)\mid\eta \text{ strictly increasing}\}.$$

\begin{defn}
If $\alpha<\theta<\kappa$ and $\alpha,\theta,\gamma\neq 0$, let $\{Z^{\alpha,\theta}_\gamma|\gamma<\kappa\}$ be an enumeration of all downward closed subtrees of $R(\alpha,\theta,i)$ for all $i$, in such a way that each possible coloured tree appears cofinally often in the enumeration. Let $Z^{0,0}_0$ be the tree $(R_f,r_f)$.
\end{defn}

This enumeration is possible because there are at most\\ $|\bigcup_{i<\omega}\mathcal{P}(R(\alpha,\theta,i))|\leq \omega\times\kappa=\kappa$ downward closed coloured subtrees. Since for all $\theta<\kappa$, $|R(\alpha,\theta,i)|<\kappa$ there are at most $\kappa\times \kappa^{<\kappa}=\kappa$ coloured trees.

\begin{defn}\label{colorconst}
Define for each $f\in \beta^\kappa$ the coloured tree $(J_f,c_f)$ by the following construction.
For every $f\in \beta^\kappa$ define $J_f=(J_f,c_f)$ as the tree of all $\eta: s\rightarrow \omega\times \kappa^4$, where $s\leq \omega$, ordered by endextension, and such that the following conditions hold for all $i,j<s$:\\
Denote by $\eta_i$, $1<i<5$, the functions from $s$ to $\kappa$ that satisfies, $$\eta(n)=(\eta_1(n),\eta_2(n),\eta_3(n),\eta_4(n),\eta_5(n)).$$
\begin{enumerate}
\item $\eta\restriction_n\in J_f$ for all $n<s$.
\item $\eta$ is strictly increasing with respect to the lexicographical order on $\omega\times \kappa^4$.
\item $\eta_1(i)\leq \eta_1(i+1)\leq \eta_1(i)+1$.
\item $\eta_1(i)=0$ implies $\eta_2(i)=\eta_3(i)=\eta_4(i)=0$.
\item $\eta_1(i)<\eta_1(i+1)$ implies $\eta_2(i+1)\ge \eta_3(i)+\eta_4(i)$.
\item $\eta_1(i)=\eta_1 (i+1)$ implies $\eta_k (i)=\eta_k (i+1)$ for $k\in \{2,3,4\}$.
\item If for some $k<\omega$, $[i,j)=\eta_1^{-1}\{k\}$, then $$\eta_5\restriction_{[i,j)}\in Z^{\eta_2(i),\eta_3(i)}_{\eta_4(i)}.$$
\noindent Note that 7 implies $Z^{\eta_2(i),\eta_3(i)}_{\eta_4(i)}\subset R(\alpha,\theta,i)$ 
\item If $s=\omega$, then either 
\begin{itemize}
\item [(a)] there exists a natural number $m$ such that $\eta_1(m-1)<\eta_1(m)$, for every $k \ge m$ $\eta_1(k)=\eta_1(k+1)$, and the color of $\eta$ is determined by $Z^{\eta_2(m),\eta_3(m)}_{\eta_4(m)}$: $$c_f(\eta)=c(\eta_5\restriction_{[m,\omega)})$$ where $c$ is the coloring function of $Z^{\eta_2(m),\eta_3(m)}_{\eta_4(m)}$.\\
\end{itemize}
or
\begin{itemize}
\item [(b)] there is no such $m$ and then $c_f(\eta)=f(sup(rng(\eta_5)))$.
\end{itemize}
\end{enumerate}
\end{defn}

Notice that for every $f\in \beta^\kappa$ and $\delta<\kappa$ with $cf(\delta)=\omega$, there is $\eta\in J_f$ such that $rng(\eta_1)=\omega$ and $\eta_5$ is cofinal to $\delta$. This $\eta$ can be constructed by taking $\langle \xi(i)\mid i<\omega\rangle$ a cofinal sequence to $\delta$, let $\eta_1=id$; let $\eta_2$, $\eta_3$, and $\eta_4$ be such that for every $i<\omega$, $\xi\restriction \{i\}\in Z^{\eta_2(i),\eta_3(i)}_{\eta_4(i)}$. Finally let $\eta_5\restriction \{i\}=\xi\restriction \{i\}$. It is clear that $\eta\in J_f$, $rng(\eta_1)=\omega$, and $\eta_5$ is cofinal to $\delta$. In particular this $\eta$ satisfies $c_f(\eta)=f(\delta)$.

\begin{fact}[Hyttinen-Kulikov, \cite{HK}, Hyttinen-Moreno, \cite{HM}]\label{colorisom}
For every $f,g\in \beta^\kappa$ the following holds $$f\ =^\beta_\omega\ g \Leftrightarrow  J_f\cong J_g$$
\end{fact}

The previous fact is an important step in \cite{HK} and in \cite{HM} to construct a reductions from $=^2_\omega$ to the isomorphism relation of different stable unsuperstable theories. We will use the coloured trees $J_f$ to construct ordered coloured trees.
Before we start with the construction of the ordered coloured trees, let us prove an important property of the coloured trees.

\begin{lemma}\label{nonzero}
For every $f\in \beta^\kappa$, $\theta<\beta$, and $\eta\in (J_f)_{<\omega}$, there is $\xi\in (J_f)_\omega$ such that $\eta<\xi$ and $c_f(\xi)=\theta$.
\end{lemma}

\begin{proof}
Let $f\in \beta^\kappa$ and $C\subseteq \kappa$ be a club. Let us define $g\in \beta^\kappa$ by: 
$$g(\alpha)=\begin{cases} f(\alpha) &\mbox{if } \alpha\in \acc (\acc (C))\\
\theta & \mbox{in other case. } \end{cases}$$
Since $\acc (\acc (C))$ is a club, $f\ =^\beta_\omega\ g$ and $J_f\cong J_g$. Therefore it is enough to show that for any $\eta\in (J_g)_{<\omega}$, there is $\xi\in (J_g)_\omega$ such that $\eta<\xi$ and $c_g(\xi)=\theta$.

Let $\eta\in (J_g)_{<\omega}$, $\eta(n)=(\eta_1(n),\eta_2(n),\eta_3(n),\eta_4(n),\eta_5(n))$, and denote by $U$ the set $\acc (C)\backslash \acc (\acc (C))$. It is clear that $U$ is unbounded, let $\delta\in U$ be such that $cf(\delta)=\omega$ and let $\langle\alpha_i\mid i<\omega\rangle$ be a sequence of elements of $C$ such that $\bigcup_{i<\omega}\alpha_i=\delta$, $sup (rng(\eta_3))+sup (rng(\eta_4))<\alpha_0$, and $sup (rng(\eta_5))<\alpha_0$. Let us construct $\xi$ in an inductive way:
\begin{itemize}
\item $\xi\restriction dom(\eta)=\eta$;
\item if $dom(\eta)\leq n<\omega$, 
\begin{itemize}
\item $\xi_1(n)=\xi_1(n-1)+1$;
\item $\xi_2(n)=\alpha_r$, where $$r=min\{i<\omega\mid\alpha_i>\xi_3(n-1)+\xi_4(n-1)\};$$
\item $\xi_3(n)=\xi_2(n)+1$;
\item $\xi_4(n)=\gamma_n$, the least ordinal such that $$Z^{\xi_2(n),\xi_3(n)}_{\gamma_n}=\{\zeta:[n,n+1)\rightarrow[\alpha_r,\alpha_r+1)\};$$
\item  $\xi_5(n)=\xi_2(n)$;
\end{itemize}
\end{itemize}
By the way we defined $\xi$, we know that $\xi\in J_g$ and $\eta\prec\xi$. By the item (8) on the construction of $J_g$, we know that $c_g(\xi)=g(sup(rng(\xi_5)))=g(\delta)$. Since $\delta\notin acc(acc(C))$, $c_g(\xi)=g(\delta)=\theta$ as we wanted.
\end{proof}

Notice that for any $f,g\in \beta^\kappa$, $J_f$ and $J_g$ are isomorphic as trees but not as colored trees. This is because $f$ is only used to define the dolor function of $J_f$.

\subsection{Construction of ordered coloured trees}

For each $f\in \beta^{\kappa}$ we will use the coloured trees $J_f$ to construct ordered coloured trees, which will be the base for the construction of the models in Section \ref{Models_section}.

Let us define the following subtrees $$J_f^\alpha=\{\eta\in J_f\mid \exists\theta<\alpha~(rng(\eta)\subset \omega\times\theta^4)\}.$$ 
Notice that $J_f^0=\{\emptyset\}$ and $dom(\emptyset)=0$.
Let us denote by $Acc(\kappa)=\{\alpha<\kappa\mid \alpha=0~ \text{or}~ \alpha \text{ is a limit ordinal}\}$. For all $\alpha\in Acc(\kappa)$ and  $\eta\in J_f^\alpha$ with $dom(\eta)=m<\omega$ define $$W_\eta^\alpha=\{\zeta\mid dom(\zeta)=[m,s), m\leq s\leq \omega, \eta^\frown\zeta\in J_f^{\alpha+\omega}, \eta^\frown (\zeta\restriction \{m\})\notin J_f^\alpha\}.$$ 
Notice that by the way $J_f$ was constructed, for every $\eta\in J_f$ with finite domain and $\alpha<\kappa$, the set $$\{(\theta_1,\theta_2,\theta_3,\theta_4,\theta_5)\in (\omega\times
\kappa^4)\backslash (\omega\times\alpha^4)\mid \eta^\frown (\theta_1,\theta_2,\theta_3,\theta_4,\theta_5)\in J_f^{\alpha+\omega}\}$$ is either empty or has size $\omega$.
Let $\sigma_\eta^\alpha$ be an enumeration of this set, when this set is not empty.

Let us denote by $\mathcal{T}=(\kappa\times \omega\times Acc(\kappa)\times \omega\times
\kappa\times\kappa\times
\kappa\times\kappa)^{\leq\omega}$. 
For every $\xi \in \mathcal{T}$ there are functions $\{\xi_i\in \kappa^{\leq \omega}\mid 0<i\leq 8\}$ such that for all $i\leq 8$, $dom(\xi_i)=dom(\xi)$ and for all $n\in dom(\xi)$, $\xi(n)=(\xi_1(n),\xi_2(n),\xi_3(n),\xi_4(n),\xi_5(n),\xi_6(n),\xi_7(n),\xi_8(n))$. 
For every $\xi\in \mathcal{T}$ let us denote $(\xi_4,\xi_5,\xi_6,\xi_7,\xi_8)$ by $\overline{\xi}$.

\begin{defn}\label{Gammas}
For all $\alpha\in Acc(\kappa)$ and  $\eta\in \mathcal{T}$ with $\overline{\eta}\in J_f$, $dom(\eta)=m<\omega$ define
$\Gamma_\eta^\alpha$ as follows:

If $\overline{\eta}\in J_f^\alpha$, then $\Gamma_\eta^\alpha$ is the set of elements $\xi$ of $\mathcal{T}$ such that:

\begin{enumerate}
\item $\xi\restriction m=\eta,$
\item $\overline{\xi}\restriction dom(\xi)\backslash m \in W^\alpha_\eta,$
\item $\xi_3$ is constant on $dom(\xi)\backslash m,$
\item $\xi_3(m)=\alpha$,
\item for all $n\in dom(\xi)\backslash m$, let $\xi_2(n)$ be the unique $r<\omega$ such that $\sigma_{\zeta}^\alpha(r)=\overline{\xi}(n)$, where $\zeta=\overline{\xi}\restriction n$.
\end{enumerate}

If $\overline{\eta}\notin J_f^\alpha$, then $\Gamma_\eta^\alpha=\emptyset$.

\end{defn}

Notice that $\xi_2(n)$ and $\xi_3(n)$ can be calculated from $\overline{\xi}\restriction n+1$.

For $\eta\in \mathcal{T}$ with $\overline{\eta}\in J_f$, $dom(\eta)=m<\omega$ define $$\Gamma(\eta)=\bigcup_{\alpha\in Acc(\kappa)}\Gamma_\eta^\alpha .$$
Finally we can define $A^f$ by induction. Let $T_f(0)=\{\emptyset\}$ and for all $n<\omega$, $$T_f(n+1)=T_f(n)\cup\bigcup_{\eta\in T_f(n)~dom(\eta)=n}\Gamma(\eta),$$ for $n=\omega$, $$T_f(\omega)=\bigcup_{n<\omega}T_f(n).$$

For  $0<i\leq 8$ let us denote by $s_i(\eta)=sup\{\eta_i(n)\mid n<\omega\}$ and $s_\omega(\eta)=sup\{s_i(\eta)\mid i\leq 8\}$, finally $$A^f=T_f(\omega)\cup \{\eta\in \mathcal{T}\mid  dom(\eta)=\omega, \forall m<\omega (\eta\restriction m\in T_f(\omega))\}.$$ 
Define the color function $d_f$ by $d_f(\eta)=c_f(\overline{\eta})$ if $s_1(\eta)< s_\omega(\eta)$ and $d_f(\eta)=f(s_1(\eta))$ otherwise.
It is clear that $A^f$ is closed under initial segments, indeed the relations $\prec$, $(P_n)_{n\leq\omega}$, and $h$ of Definition \ref{Shtree} have a canonical interpretation in $A^f$. 

Now we finish the construction of $A^f$ by using the $\kappa$-colorable linear order $I$. 
We only have to define $<\restriction Suc_{A^f}(\eta)$ for all $\eta\in A^f$ with finite domain. 
Properly speaking, $A^f$ will not be an ordered coloured tree as in Definition \ref{Shtree}, but it will be isomorphic to an ordered coloured tree as in Definition \ref{Shtree}. 

Let us proceed to define $<\restriction Suc_{A^f}(\eta)$.  Let $F:I\rightarrow\kappa$ be a $\kappa$-coloration of $I$.
 
For any $\eta\in A^f$ with domain $m<\omega$, we will define the order $<\restriction Suc_{A^f}(\eta)$ such that it is isomorphic to $I$ and satisfies the following:

$(\ast)$
\textit{
 For any set $B\subset Suc_{A^f}(\eta)$ of size less than $\kappa$, $p=tp_{bs}(x,B,Suc_{A^f}(\eta))$%
 , and any tuple $(\theta_2,\theta_3)\in \omega\times Acc(\kappa)$ with $\theta_3\ge \eta_3(m-1)$, if $p$ is realized in $Suc_{A^f}(\eta)$, then
 there are $\kappa$ many $\gamma<\kappa$ such that $\eta^\frown (\gamma,\theta_2,\theta_3,\sigma^{\theta_3}_{\overline{\eta}}(\theta_2))\models p$.
}

By the construction of $A^f$, an isomorphism between $\{(\theta_1,\theta_2,\theta_3)\in \kappa\times \omega\times Acc(\kappa)\mid \theta_3\ge \eta_3(m-1) \}$ and $I$, induces an order in $Suc_{A^f}(\eta)$.

\begin{defn}

Recall $F$ the $\kappa$-coloration of $I$ in Theorem \ref{A_desire_order}. For all $\theta, \alpha<\kappa$, let fix bijections $\tilde{G}_\theta:\{(\theta_2,\theta_3)\in \omega\times Acc(\kappa) \mid \theta_3\ge \theta\}\rightarrow \kappa$ and $\tilde{H}_\alpha:F^{-1}[\alpha]\rightarrow\kappa$. 
Notice that these functions exist because $F$ is a $\kappa$-coloration of $I$ and there are $\kappa$ tuples $(\theta_2,\theta_3)$.

Let us define $\tilde{\mathcal{G}}_\theta:\{(\theta_1,\theta_2,\theta_3)\in \kappa\times \omega\times Acc(\kappa) \mid \theta_3\ge \theta\}\rightarrow I$, by $\tilde{\mathcal{G}}_\theta((\theta_1,\theta_2,\theta_3))=a$ where $a$ is the unique element that satisfies:
\begin{itemize}
\item $\tilde{G}_\theta((\theta_2,\theta_3))=\alpha$;
\item $\tilde{H}_\alpha(a)=\theta_1$.
\end{itemize} 

\end{defn}

For any $\eta\in A^f$ with domain $m<\omega$ and $\eta_3(m-1)=\theta$, the isomorphism $\tilde{\mathcal{G}}_\theta$ induces an order in $Suc_{A^f}(\eta)$. Let us define $<\restriction Suc_{A^f}(\eta)$ as the induced order given by  $\tilde{\mathcal{G}}_\theta$.

\begin{fact}
Suppose $\eta\in A^f$ has domain $m<\omega$ and $\eta_3(m-1)=\theta$. Then
$<\restriction Suc_{A^f}(\eta)$ satisfies ($*$). 
\end{fact}
\begin{proof}
Let $b\in Suc_{A^f}(\eta)$, $(\theta_2,\theta_2)\in \omega\times Acc(\kappa)$ such that $\theta_3\ge \eta_3(m-1)=\theta$, and $B\subseteq Suc_{A^f}(\eta)$ have size less than $\kappa$. Let us denote by $q$ the type $tp_{bs}(\tilde{\mathcal{G}}_\theta(b_1,b_2,b_3), \tilde{\mathcal{G}}_\theta (B\cap(\kappa\times\omega\times Acc(\kappa))), I)$. By the construction of $\tilde{G}_\theta$, since $F$ is a $\kappa$-coloration of $I$, $$|\{a\in I\mid a\models q \ \&\ F(a)=\tilde{G}_\theta(\theta_2,\theta_3)\}|=\kappa.$$
Therefore for all $a$ such that $a\models q$ and $F(a)=\tilde{G}_\theta(\theta_2,\theta_3)$, $$\eta^\frown (\tilde{H}_{\tilde{G}_\theta((\theta_2,\theta_3))}(a),\theta_2,\theta_3,\sigma^{\theta_3}_{\overline{\eta}}(\theta_2))\models p$$
\end{proof}

It is clear that $(A^f,\prec,(P_n)_{n\leq\omega},<, h)$ is isomorphic to a subtree of $I^{\leq \omega}$ in the sense of Definition \ref{Shtree}.

\begin{remark}\label{equal_orders}
Notice that for any $\eta\in A^f$, $<\restriction Suc_{A^f}(\eta)$ is isomorphic to $I$. Therefore for any $\zeta,\eta\in A^f$, $<\restriction Suc_{A^f}(\zeta)$ and $<\restriction Suc_{A^f}(\eta)$ are isomorphic. Even more, the construction of $<\restriction Suc_{A^f}(\eta)$ only depends on $\eta_3(m-1)$, where $m<\omega$ is the domain of $\eta$.
\end{remark}

\begin{remark}\label{equal_trees}
Same as in the construction of the colored trees $J_f$, the function$f\in \beta^\kappa$ is only used to define the color function in the construction of $A^f$. 
So if $f,g\in \beta^\kappa$ and $\alpha$ are such that $f\restriction\alpha=g\restriction\alpha$, then $J_f^\alpha=J_g^\alpha$. As a consequence  $f\restriction\alpha=g\restriction\alpha$ implies that $A^f_\alpha= A^g_\alpha$.
\end{remark}

Notice that the only property we used from $I$ to construct the ordered coloured trees was that it is a  $\kappa$-colorable linear order. Therefore the construction can be done with any $\kappa$-colorable linear order.
  
\begin{thm}\label{Afcong}
Suppose $f,g\in \beta^\kappa$, then $f\ =^\beta_\omega\ g$ if and only if $A^f\cong A^g$ (as ordered coloured trees).
\end{thm}

\begin{proof}
For every $f\in \beta^\kappa$ let us define the $\kappa$-representation $\mathbb{A}^f=\langle A^f_\alpha \mid \alpha<\kappa\rangle$ of $A^f$, 
$$
A^f_\alpha=\{\eta\in A^f\mid rng(\eta)\subseteq \theta\times\omega\times\theta\times\omega\times\theta^4\text{ for some }\theta<\alpha\}.
$$

Let $f$ and $g$ be such that $f\ =^\beta_\omega\ g$, there is $G$ a coloured trees isomorphism between $J_f$ and $J_g$. Let $C\subseteq \kappa$ be a club such that $\{\alpha\in C\mid cf(\alpha)=\omega\}\subseteq \{\alpha<\kappa\mid f(\alpha)=g(\alpha)\}$. 
We will show that there are sequences $\{\alpha_i\}_{i<\kappa}$ and $\{F_i\}_{i<\kappa}$ with the following properties:
\begin{itemize}
\item $\{\alpha_i\}_{i<\kappa}$ is a club;
\item if $i$ is a successor, then there is $\theta\in C$ such that $\alpha_{i-1}<\theta<\alpha_i$;
\item if $i=\gamma+n$ and $n$ is odd, $F_i$ is a partial isomorphism between $A^f$ and $A^g$, and $A_{\alpha_i}^f\subseteq dom(F_i)$;
\item if $i=\gamma+n$ and $n$ is even, $F_i$ is a partial isomorphism between $A^f$ and $A^g$, and $A_{\alpha_i}^g\subseteq rng(F_i)$;
\item if $i$ is limit, then $F_i:A_{\alpha_i}^f\rightarrow A_{\alpha_i}^g$;
\item if $i<j$, then $F_i\subseteq Fj$;
\item for all $\eta\in dom(F_i)$, $G(\overline{\eta})=\overline{F_i(\eta)}$.
\end{itemize}

We will proceed by induction over $i$, for the case $i=0$, let $\alpha_0=0$ and $F_0(\emptyset)=\emptyset$. Suppose $i=\gamma+n$ with $n$ even is such that $F_i$ is a partial isomorphism, $A_{\alpha_i}^g\subseteq rng(F_i)$ for all $j<i$, $F_j\subseteq F_i$, and $G(\overline{\eta})=\overline{F_i(\eta)}$ for all $\eta\in dom(F_i)$.

Let us choose $\alpha_{i+1}$ be a successor ordinal such that $\alpha_{i}<\theta<\alpha_{i+1}$ holds for some $\theta\in C$ and enumerate $A^f_{\alpha_{i}}$ by $\{\eta_j\mid j<\Omega\}$ for some $\Omega<\kappa$. Denote by $B_j$ the set $\{x\in A^f_{\alpha_{i+1}}\backslash dom(F_i)\mid \eta_j\prec x\}$. 

By the induction hypothesis, we know that for all $j<\Omega$, $x\in B_j$, $\overline{F_i(\eta_j)}\prec G(\overline{x})$. By Remark \ref{equal_orders}, for all $\eta\in A^f$ and $\xi\in A^g$, $<\restriction Suc_{A^f}(\eta)$ and $<\restriction Suc_{A^g}(\xi)$ are isomorphic. Thus, since $|A^f_{\alpha_{i}}|, |B_0|<\kappa$, by $(\ast)$ there is an embedding $F^0_i$ from $(A^f_{\alpha_i}\cup B_0,\prec,<)$ to $(A^g,\prec, <)$ that extends $F_i$ and for all $\eta\in dom(F^0_i)$, $\overline{F^0_i(\eta)}=G(\overline{\eta})$. 

For the case $B_j$ for $j>0$, let us suppose that $t<\Omega$ is such that the following hold:
\begin{itemize}
\item there is a sequence of embeddings $\{F_i^j\mid j<t\}$, where $F^j_i$ is an embedding from $(A^f_{\alpha_i}\cup \bigcup_{l\leq j}B_l,\prec,<)$ into $A^g$,
\item $F_i^l\subseteq F_i^j$ holds for all $l<j<t$,
\item for all $\eta\in dom(F^j_i)$, $\overline{F^j_i(\eta)}=G(\overline{\eta})$.
\end{itemize}

Since $|A^f_{\alpha_{i}}\cup \bigcup_{j< t}B_j|, |B_t|<\kappa$, by $(\ast)$ there is an embedding $F^t_i$ from $(A^f_{\alpha_i}\cup \bigcup_{j\leq t}B_j,\prec,<)$ to $(A^g,\prec, <)$ that extends $\bigcup_{j<t} F_i^j$ and for all $\eta\in dom(F^t_i)$, $\overline{F^t_i(\eta)}=G(\overline{\eta})$. 

Finally if $i$ is not a limit, then $F_{i+1}=\bigcup_{j<\Omega} F_i^j$ is as wanted. Otherwise, for all $x\in dom(\bigcup_{j<\Omega} F_i^j)$, $F_{i+1}(x)=\bigcup_{j<\Omega} F_i^j~(x)$ and for all $x\in A^f_{\alpha_{i+1}}\backslash A^f_{\alpha_i}$ such that for all $n<\omega$, $x\restriction n\in A^f_{\alpha_i}$; we define $F_{i+1}(x)$ be the unique $y\in A^g_{\alpha_{i+1}}$ such that for all $n<\omega$, $F_{i}(x\restriction n)\prec y$. Thus $F_{i+1}$ is as wanted.

The case $i=\gamma+n$ with $n$ odd is similar. For $i$ limit, we define $\alpha_i=\bigcup_{j<i} \alpha_j$ and $F_{\alpha_{i}}=\bigcup_{j<i} F_j$.

It is clear that $F=\bigcup_{j<\kappa} F_j$ witnesses that $A^f$ and $A^g$ are isomorphic as ordered trees. Let us show that $d_f(\eta)=d_g(F(\eta))$, suppose $\eta\in A^f$ is a leaf. Let $l$ be the least ordinal such that $\eta\in A^f_{\alpha_l}$. If there is $n<\omega$ such that for all $j<l$, $\eta\restriction n\notin A^f_{\alpha_j}$, then by the way $F$ was constructed, $d_f(\eta)=d_g(F(\eta))$. On the other hand, if for all $n<\omega$ there is $j<l$ such that $\eta\restriction n\in A^f_{\alpha_j}$, then there is an $\omega$-cofinal ordinal $i$ such that $s_\omega=\alpha_i$ and $i+1=l$. By the construction of $A^f$ we know that either $d_f(\eta)=f(s_1)$ (if $s_1=\alpha_i$) or $d_f(\eta)=c_f(\overline{\eta})$ (if $s_1<\alpha_i=s_8$). 
From the definition of $J_f$, if $\overline{\eta}\in (J_f)_\omega\backslash J_f^{\alpha_i}$ and for all $n<\omega$, $\overline{\eta}\restriction n\in J_f^{\alpha_i}$, then $c_f(\overline{\eta})=f(s_8)$. Therefore $d_f(\eta)=f(\alpha_i)$ holds in both cases ($s_1=s_\omega$ and $s_1<s_\omega$). By the same argument and using the definition of $F$, we can conclude that $d_g(F(\eta))=g(\alpha_i)$. Finally since $i$ is a limit ordinal with cofinality $\omega$, $\alpha_i$ is an $\omega$-limit of $C$. Thus $d_f(\eta)=f(\alpha_i)=g(\alpha_i)=d_g(F(\eta))$ and $F$ is a coloured tree isomorphism.

Now let us prove that if $A^f$ and $A^g$ are isomorphic ordered coloured trees, then $f~ =^\beta_\omega~ g$.

Let us start by defining  the following function $H_f\in \beta^\kappa$. For every $\alpha\in \kappa$ with cofinality $\omega$, define $B_\alpha=\{\eta\in A^f\backslash A^f_\alpha\mid dom(\eta)=\omega~\wedge~\forall n<\omega ~(\eta\restriction n\in A^f_\alpha)\}$. Notice that by the construction of $A^f$ and the definition of $A^f_\alpha$, for all $\eta\in B_\alpha$ we have $d_f(\eta)=f(s_\omega)=f(\alpha)$. Therefore, the value of $f(\alpha)$ can be obtained from $B_\alpha$ and $d_f$, and we can define the function $H_f\in \beta^\kappa$ as :

$$H_f(\alpha)=\begin{cases} f(\alpha) &\mbox{if } cf(\alpha)=\omega\\
0 & \mbox{in other case. } \end{cases}$$
This function can be obtained from the $\kappa$-representation $\{A^f_\alpha\}_{\alpha<\kappa}$ and $d_f$. It is clear that $f~ =^\beta_\omega~ H_f$.

\begin{claim}
If $A^f$ and $A^g$ are isomorphic ordered coloured trees, then $H_f~ =^\beta_\omega~ H_g$.
\end{claim}
\begin{proof}
Let $F$ be an ordered coloured tree isomorphism. It is easy to see that $\{F[A^f_{\alpha}]\}_{\alpha<\kappa}$ is a $\kappa$-representation. Define $C=\{\alpha<\kappa\mid F[A^f_\alpha]=A^g_\alpha\}$. Since $F$ is an isomorphism, for all $\alpha\in C$, $H_f(\alpha)=H_g(\alpha)$. Therefore it is enough to show that $C$ is $\omega$-closed and unbounded. By the definition of $\kappa$-representation, if $(\alpha_n)_{n<\omega}$ is a sequence of elements of $C$ cofinal to $\gamma$, then $A^g_\gamma=\bigcup_{n<\omega}A^g_{\alpha_n}=\bigcup_{n<\omega}F[A^f_{\alpha_n}]=F[A^f_\gamma]$. We conclude that $C$ is $\omega$-closed. 

Let us finish by showing that $C$ is unbounded. Fix an ordinal $\alpha<\kappa$, let us construct a sequence $(\alpha_n)_{n\leq\omega}$ such that $\alpha_\omega\in C$ and $\alpha_\omega>\alpha$. Define $\alpha_0=\alpha$. For every odd $n$, define $\alpha_{n+1}$ to be the least ordinal bigger than $\alpha_n$ such that $F[A^f_{\alpha_n}]\subseteq A^g_{\alpha+1}$.  For every even $n$, define $\alpha_{n+1}$ to be the least ordinal bigger than $\alpha_n$ such that $A^g_{\alpha_n}\subseteq F[A^f_{\alpha+1}]$. Define $\alpha_\omega=\bigcup_{n<\omega}\alpha_n$. Clearly $\bigcup_{i<\omega}F[A^f_{\alpha_{2i}}]=\bigcup_{i<\omega}A^g_{\alpha_{2i+1}}$. We conclude that $\alpha_\omega\in C$
\end{proof}
\end{proof}

Notice that the only property of $<\restriction Suc_{A^f}(\eta)$ that we used in the previous theorem was ($*$). Therefore, the previous theorem can be generalized to the following corollary.

\begin{cor}
Suppose $l$ is a $\kappa$-colorable linear order and $\beta\leq\kappa$. Then for any $f\in \beta^\kappa$, there is an ordered coloured tree $A^f(l)$ that satisfies: For all $f,g\in \beta^\kappa$,
$$f\ =^\beta_\omega\ g\ \Leftrightarrow A^f(l)\cong A^g(l).$$
\end{cor}


\section{The Models}\label{Models_section}

\subsection{Generalized Ehrenfeucht-Mostowski models}

In this section we will use the generalized Ehrenfeucht-Mostowski models (see \cite{Sh90} Chapter VII. 2 or \cite{HT} Section 8) to construct the models of unsuperstable theories, we will use the previous constructed ordered coloured trees (from $I$) as the skeleton of the construction.

\begin{defn}[Generalized Ehrenfeucht-Mostowski models]\label{EM-models_def}
We say that a function $\Phi$ is proper for $K_{tr}^\gamma$, if there is a vocabulary $\mathcal{L}^1$ and for each $A\in K_{tr}^\gamma$, there is a model $\mathcal{M}_1$ and tuples $a_s$, $s\in A$, of elements of $\mathcal{M}_1$ such that the following two hold:
\begin{itemize}
\item every element of $\mathcal{M}_1$ is an interpretation of some $\mu(a_s)$, where $\mu$ is a $\mathcal{L}^1$-term;
\item $tp_{at}(a_s, \emptyset, \mathcal{M}_1)=\Phi(tp_{at}(s, \emptyset, A))$.
\end{itemize}
Notice that for each $A$, the previous conditions determine $\mathcal{M}_1$ up to isomorphism. We may assume $\mathcal{M}_1$, $a_s$, $s\in A$, are unique for each $A$. We denote $\mathcal{M}_1$ by $EM^1(A,\Phi)$. We call $EM^1(A,\Phi)$ an \textit{Ehrenfeucht-Mostowski model}.
\end{defn}

Suppose $T$ is a countable complete theory in a countable vocabulary $\mathcal{L}$, $\mathcal{L}^1$ a Skolemization of $\mathcal{L}$, and $T^1$ the Skolemization of $T$ by $\mathcal{L}^1$. If there is $\Phi$ a proper function for $K_{tr}^\lambda$, then for every $A\in K_{tr}^\gamma$, we will denote by EM$(A,\Phi)$ the $\mathcal{L}$-reduction of $EM^1(A,\Phi)$. 
The following result ensure the existence  of a proper function $\Phi$ for unsuperstable theories $T$ and $\gamma=\omega$.

\begin{fact}[Shelah, \cite{Sh} Theorem 1.3, proof in \cite{Sh90} Chapter VII 3]
Suppose $\mathcal{L}\subseteq \mathcal{L}^1$ are vocabularies, $T$ is a complete first order theory in $\mathcal{L}$, $T^1$ is a complete theory in $\mathcal{L}^1$ extending $T$ and with Skolem-functions. Suppose $T$ is unsuperstable and $\{\phi_n(x,y_n)\mid n<\omega\}$ witnesses this. Then there is a function $\Phi$ proper such that for all $A\in K_{tr}^\omega$, $EM^1(A,\Phi)$ is a model of $T^1$, and for $s\in P^A_n$, $t\in P^A_\omega$, $EM^1(A,\Phi)\models \phi_n(a_t,a_s)$ if and only if $A\models s\prec t$.
\end{fact}

The models that we will construct are of the form $EM(A,\Phi)$. 

\subsection{ Reduction of the Isomorphism Relation}

Before we deal with the construction of the models and the reduction, we need to do some preparations.

\begin{defn}
For any $A\in K^\omega_{tr}$ with size $\kappa$ and $\mathbb{A}$ a $\kappa$-representation of $A$, we define $S(\mathbb{A})$ as the set $$\{\delta<\kappa\mid \delta\text{ a limit ordinal, } \exists \eta\in P^A_\omega, \{\eta\restriction n\mid n<\omega\}\subseteq A_\delta\ \wedge\ \forall \alpha<\delta (\{\eta\restriction n\mid n<\omega\}\not\subseteq A_\alpha)\}$$
\end{defn}

\begin{fact}[Shelah, \cite{Sh} Fact 2.3, Hyttinen-Tuuri, \cite{HT} Lemma 8.6,]\label{fact2.8}
$S$ is a $CUB$-invariant function.
\end{fact}

This fact allows us to define $S(A)$ for $A\in K_{tr}^\omega$ as $\big[ S(\mathbb{A})\big]_{=^2_{CUB}}$ for any $\mathbb{A}$ $\kappa$-representation of $A$.

Clearly for any of the ordered coloured trees $A^f$, the ones constructed in Section \ref{section_ordered_colored_trees}, we have that $S(A^f)$ is the set of limit ordinals. This is because all the branches of $A^f$ have order type $\omega+1$. This can be fixed easily by restricting ourselves to the generalized Cantor space $f\in 2^\kappa$.

\begin{defn}\label{the_models}
Let $I$ be the $(<\kappa, bs)$-stable $(\kappa, bs,bs)$-nice $\kappa$-colorable linear order from Section \ref{section_colorable_linear_order}. For every $f\in 2^\kappa$, let $A^f$ be the tree constructed in Section \ref{section_ordered_colored_trees}. 
Define the tree $A_f\subseteq A^f$ by: $x\in A_f$ if and only if $x$ is not a leaf of $A^f$ or $x$ is a  leaf such that $d_f(x)=1$.
Denote by $\mathcal{A}^f$ the model EM$(A_f,\Phi)$.
\end{defn}

Since $I$ is $(\kappa, bs,bs)$-nice, by Lemma \ref{nonzero} the trees $A_f$ are locally $(\kappa, bs,bs)$-nice. Notice that since the branches of the trees $A_f$ have length at most $\omega+1$ and $I$ is $(<\kappa, bs)$-stable, then the trees $A_f$ are $(<\kappa, bs)$-stable.

By the way the models EM$(A,\Phi)$ were define, we know that if $A,A'\in K_{tr}^\omega$ are isomorphic, then  EM$(A,\Phi)$ and EM$(A',\Phi)$ are isomorphic. Thus if $A_f$ and $A_g$ are isomorphic, then $\mathcal{A}^f$ and $\mathcal{A}^g$ are isomorphic.
 
Notice that since we are working under the assumption $\kappa$ is an uncountable cardinal satisfying $\kappa^{<\kappa}=\kappa$, $\kappa>|\mathcal{L}^1|$.

\begin{lemma}\label{spectra}
For every $f,g\in 2^\kappa$:

$$f\ =^2_\omega\ g \text{ if and only if }S(A_f)=S(A_g).$$
\end{lemma}

\begin{proof}
By Fact \ref{fact2.8}, $S$ is $CUB$-invariant, therefore it is enough to find a $\kappa$-representation $\mathbb{A}_f$ of $A_f$ for every $f\in 2^\kappa$, such that for all $f,g\in 2^\kappa$, $f\ =^2_\omega\ g$ if and only $\mathbb{A}_f\ =^2_{CUB}\ \mathbb{A}_g$.

Similar as in the proof of Theorem \ref{Afcong}, for all $f\in 2^\kappa$ let us define the $\kappa$-representation $\mathbb{A}_f=\langle A_{f,\alpha}\mid \alpha<\kappa \rangle$ by

$$
A_{f,\alpha}=\{\eta\in A_f\mid rng(\eta)\subseteq \theta\times\omega\times\theta\times\omega\times\theta^4\text{ for some }\theta<\alpha\}.
$$

By definition $$S(\mathbb{A}_f)=\{\delta<\kappa\mid \exists\eta\in P_\omega^{A_f}, \{\eta\restriction n\mid n<\omega\}\subseteq (A_{f,\delta}\ \&\ \forall \alpha<\delta(\{\eta\restriction n\mid n<\omega\}\not\subseteq A_{f,\alpha})\}.$$

\begin{claim}
$\delta\in S(\mathbb{A}_f)$ if and only if $cf(\delta)=\omega$ and for every $\eta\in P_\omega^{A_f}$ that witnesses it,  $sup(\{rng(\eta_i)\mid i\leq 8\})=\delta$ .
\end{claim}

\begin{proof}
The direction from right to left follows from Definition \ref{the_models}. The other direction follows from the definition of $S(\mathbb{A}_f)$ and $A_{f,\alpha}$.
\end{proof}

By the previous Claim we know that if $\delta\in S(\mathbb{A}_f)$ and $\eta\in  P_\omega^{A_f}$ witnesses it, then by the way $A^f$ was constructed, $1=d_f(\eta)=f(sup(rng(\eta_1), rng(\eta_8)))=f(\delta)$. Therefore we can rewrite $S(\mathbb{A}_f)$ as $$S(\mathbb{A}_f)=\{\delta<\kappa\mid cf(\delta)=\omega \wedge f(\delta)=1\}.$$
It follows that $S(\mathbb{A}_f)\ =^2_{CUB}\ S(\mathbb{A}_g)$ holds if and only if $f\ =^2_\omega\ g$.
\end{proof}

Now we proceed to prove that the models $\mathcal{A}^f$ are as wanted, i.e. $f\ =^2_\omega\ g$ if and only if $\mathcal{A}^f\ \cong_T \ \mathcal{A}^g$.

\begin{fact}[Shelah, \cite{Sh} Theorem 2.4]\label{SHmain}
Suppose $T$ is a countable complete unsuperstable theory in a countable vocabulary. If $\kappa$ is a regular uncountable cardinal, $A_1, A_2\in K_{tr}^\omega$ have size $\kappa$, $A_1$, $A_2$ are locally $(\kappa, bs,bs)$-nice and $(<\kappa, bs)$-stable, EM$(A_1,\Phi)$ is isomorphic to EM$(A_2,\Phi)$, then $S(A_1)=S(A_2)$.
\end{fact}

\begin{lemma}\label{main}
If $T$ is a countable complete unsuperstable theory over a countable vocabulary, then for all $f,g\in 2^\kappa$, $f\ =^2_\omega\ g$ if and only if $\mathcal{A}^f$ and $\mathcal{A}^g$ are isomorphic.
\end{lemma}

\begin{proof}
From left to right. Suppose $f,g\in 2^\kappa$ are such that $f\ =_\omega^2\ g$.
By Theorem \ref{Afcong} and Definition \ref{the_models} we know that $f\ =_\omega^2\ g$ if and only if $A_f\cong A_g$. Finally $A_f\cong A_g$ implies that $\mathcal{A}^f$ and $\mathcal{A}^g$ are isomorphic.

From right to left. Suppose $f,g\in 2^\kappa$ are such that $\mathcal{A}^f$ and $\mathcal{A}^g$ are isomorphic. By Definition \ref{the_models} and Fact \ref{SHmain}, $S(A_f)=S(A_g)$. From Lemma \ref{spectra} we conclude $f\ =^2_\omega\ g$.
\end{proof}

\begin{thm}\label{maincor}
If $T$ is a countable complete unsuperstable theory over a countable vocabulary, $\mathcal{L}$, then $=^2_\omega\ \reduc\ \cong_T$. 
\end{thm}

\begin{proof}
For every $f\in 2^\kappa$, we will construct a model $\mathcal{M}^f$ isomorphic to EM$(A_f,\Phi)$. We will also construct a function $\mathcal{G}:\{\mathcal{M}^f\mid f\in 2^\kappa\}\rightarrow 2^\kappa$, such that $\mathcal{A}_{\mathcal{G}(\mathcal{M}^f)}\cong \mathcal{M}^f$ and $f\mapsto \mathcal{G}(\mathcal{M}^f)$ is continuous. 
By Remark \ref{equal_trees}, Definition \ref{the_models}, and the definition of $A_{f,\alpha}$, 
$$f\restriction \alpha=g\restriction \alpha \Leftrightarrow A_{f,\alpha}=A_{g,\alpha}.$$ 
Let us denote by $SH(X)$ the Skolem-hull of $X$, i.e. $\{\mu(a)\mid a\in X, \mu\text{ an }\mathcal{L}^1 \text{-term}\}$. For all $\alpha$, $A\in K^\omega_{tr}$, and a $\kappa$-representation $\mathbb{A}=\langle A_\alpha\mid \alpha<\kappa  \rangle$ of $A$, let us denote by $\tilde{A}_\alpha$ the set $\{a_s\mid s\in A_\alpha\}$, recall the construction of $EM^1(A,\Phi)$ in Definition  \ref{EM-models_def}. 
Since for all $\alpha<\kappa$, $$A_{f,\alpha}=A_{g,\alpha} \Leftrightarrow SH(\tilde{A}_{f,\alpha}) = SH(\tilde{A}_{g,\alpha}),$$
we can construct for all $f$ a tuple $(\mathcal{M}^f, F_f)$, where $\mathcal{M}^f$ is a model isomorphic to EM$(A_f,\Phi)$ and $F_f:\mathcal{M}^f\rightarrow \text{ EM}(A_f,\Phi)$ is an isomorphism, that satisfies the following: denote by $\mathcal{M}^f_\alpha$ the preimage $F^{-1}_f[SH(\tilde{A}_{f,\alpha})\restriction \mathcal{L}]$ and $$f\restriction \alpha=g\restriction \alpha \Leftrightarrow \mathcal{M}^f_\alpha=\mathcal{M}^g_\alpha.$$
For every $f\in 2^\kappa$ there is a bijection $E_f: dom(\mathcal{M}^f)\rightarrow \kappa$, such that for every $f,g\in 2^\kappa$ and $\alpha<\kappa$ it holds that: If $f\restriction \alpha=g\restriction \alpha $, then $E_f\restriction dom(\mathcal{M}^f_\alpha)=E_g\restriction dom(\mathcal{M}^g_\alpha)$ (see \cite{Mor}). 
Let $\pi$ be the bijection in Definition \ref{struct}, define the function $\mathcal{G}$ by: $$\mathcal{G}(\mathcal{M}^f)(\alpha)=\begin{cases} 1 &\mbox{if } \alpha=\pi(m,a_1,a_2,\ldots,a_n) \text{ and }\\ & \mathcal{M}^{f}\models Q_m(E_f^{-1}(a_1),E_f^{-1}(a_2),\ldots,E_f^{-1}(a_n))\\
0 & \mbox{in other case. } \end{cases}$$
To show that $G:2^\kappa\rightarrow 2^\kappa$, $G(f)=\mathcal{G}(\mathcal{M}^f)$ is continuous, let $[\zeta\restriction \alpha]$ be a basic open set and $\xi\in G^{-1}[[\zeta\restriction \alpha]]$. There is $\beta<\kappa$ such that for all $\gamma<\alpha$, if $\gamma=\pi(m,a_1,\ldots, a_n)$, then $E^{-1}_\xi (a_i)\in dom(\mathcal{M}^\xi_\beta)$ holds for all $i\leq n$. Since for all $\eta\in [\xi\restriction\beta]$ it holds that $\mathcal{M}^\eta_\beta=\mathcal{M}^\xi_\beta$, for all $\gamma<\alpha$ that satisfies $\gamma=\pi(m,a_1,\ldots, a_n)$ $$\mathcal{M}^\eta\models Q_m(E_\eta^{-1}(a_1),E_\eta^{-1}(a_2),\ldots,E_\eta^{-1}(a_n))$$ if and only if $$\mathcal{M}^\xi\models Q_m(E_\xi^{-1}(a_1),E_\xi^{-1}(a_2),\ldots,E_\xi^{-1}(a_n)).$$
We conclude that $G$ is continuous.
\end{proof}

\subsection{Corollaries}

In this section we will prove Theorem A and Theorem B.

\begin{fact}[Hyttinen-Kulikov-Moreno, \cite{HKM} Lemma 2]\label{HKM1}
Assume $T$ is a countable complete classifiable theory over a countable vocabulary. If $\diamondsuit_\omega$ holds, then $\cong_T \ \reduc\ =^2_\omega$.
\end{fact}

\begin{fact}[Friedman-Hyttinen-Kulikov, \cite{FHK13} Theorem 77]\label{notred}
If a first order countable complete theory over a countable vocabulary $T$ is classifiable, then $=^2_\omega\ \not\reduc\ \cong_T$.
\end{fact}

\begin{cor}
Suppose $\kappa=\lambda^+=2^\lambda$ and $\lambda^\omega=\lambda$. If $T_1$ is a countable complete classifiable theory, and $T_2$ is a countable complete unsuperstable theory, then $\cong_{T_1}\ \reduc\ \cong_{T_2}$ and $\cong_{T_2}\ \not\reduc\ \cong_{T_1}$.
\end{cor}

\begin{proof}
Since $\lambda^\omega=\lambda$, $cf(\lambda)>\omega$. By \cite{Sh1} we know that if $\kappa=\lambda^+=2^\lambda$ and $cf(\lambda)>\omega$, then $\diamondsuit_\omega$ holds. The proof follows from Theorem \ref{maincor}, Fact \ref{HKM1},  and Fact \ref{notred}.
\end{proof}

We will finish this section with a corollary about $\Sigma^1_1$-completeness. Before we state the corollary we need to recall some definitions from \cite{FMR} in particular the definition of $\dl^*_S(\Pi^1_2)$. For more on $\dl^*_S(\Pi^1_2)$ see \cite{FMR}.

A $\Pi^{1}_{2}$-sentence $\phi$ is a formula of the form $\forall X\exists Y\varphi$ where $\varphi$ is a first-order sentence over a relational language $\mathcal L$ as follows:
\begin{itemize}
\item $\mathcal L$ has a predicate symbol $\epsilon$ of arity $2$;
\item $\mathcal L$ has a predicate symbol $\mathbb X$ of arity $m({\mathbb X})$;
\item $\mathcal L$ has a predicate symbol $\mathbb Y$ of arity $m({\mathbb Y})$;
\item $\mathcal L$ has infinitely many predicate symbols $(\mathbb B_n)_{n\in \omega}$, each $\mathbb B_n$ is of arity $m(\mathbb B_n)$.
\end{itemize}

\begin{defn} For sets $N$ and $x$, we say that \emph{$N$ sees $x$} iff
$N$ is transitive, p.r.-closed, and $x\cup\{x\}\s N$.
\end{defn}

Suppose that a set $N$ sees an ordinal $\alpha$,
and that $\phi=\forall X\exists Y\varphi$ is a $\Pi^{1}_{2}$-sentence, where $\varphi$ is a first-order sentence in the above-mentioned language $\mathcal L$.
For every sequence $(B_n)_{n\in\omega}$ such that, for all $n\in\omega$, $B_n\s \alpha^{m(\mathbb B_n)}$,
we write
$$\langle \alpha,{\in}, (B_{n})_{n\in \omega} \rangle \models_N \phi$$
to express that the two hold:
\begin{enumerate}
\item $(B_{n})_{n\in \omega} \in N$;
\item $\langle N,\in\rangle\models (\forall X\subseteq \alpha^{m(\mathbb X)})(\exists Y\subseteq \alpha^{m(\mathbb Y)})[\langle \alpha,{\in}, X, Y, (B_{n})_{n\in \omega}  \rangle\models \varphi]$,
where:
\begin{itemize}
\item    $\in$ is the interpretation of $\epsilon$;
\item $X$ is the interpretation of $\mathbb X$;
\item  $Y$ is the interpretation of $\mathbb Y$, and
\item for all $n\in\omega$,  $B_n$ is the interpretation of $\mathbb B_n$.
\end{itemize}
\end{enumerate}

\begin{defn}\label{reflectingdiamond}   Let $\kappa$ be a regular and uncountable cardinal, and $S\s\kappa$ stationary.

    $\dl^*_S(\Pi^1_2)$ asserts the existence of a sequence $\vec N=\langle N_\alpha\mid\alpha\in S\rangle$ satisfying the following:
    
    \begin{enumerate}
        \item for every $\alpha\in S$, $N_\alpha$ is a set of cardinality $<\kappa$ that sees $\alpha$;
        \item for every $X\s\kappa$, there exists a club $C\s\kappa$ such that, for all $\alpha\in C \cap S$, $X\cap\alpha\in N_\alpha$;
        \item whenever $\langle \kappa,{\in},(B_n)_{n\in\omega}\rangle\models\phi$,
        with $\phi$ a $\Pi^1_2$-sentence,
        there are stationarily many $\alpha\in S$ such that $|N_\alpha|=|\alpha|$ and
        $\langle \alpha,{\in},(B_n\cap(\alpha^{m(\mathbb B_n)}))_{n\in\omega}\rangle\models_{N_\alpha}\phi$.
    \end{enumerate}
 \end{defn}

\begin{fact}[Fernandes-Moreno-Rinot, \cite{FMR} Theorem C]\label{comp-rels}
If $\dl^*_S(\Pi^1_2)$ holds for $S=\{\alpha<\kappa\mid cf(\alpha)=\omega\}$, then $=^2_\omega$ is $\Sigma^1_1$-complete.
\end{fact}

\begin{cor}
If $\dl^*_S(\Pi^1_2)$ holds for $S=\{\alpha<\kappa\mid cf(\alpha)=\omega\}$, and $T$ is a countable complete unsuperstable theory, then $\cong_T$ is $\Sigma^1_1$-complete. 
\end{cor}

\begin{proof}
It follows from Fact \ref{comp-rels} and Theorem \ref{maincor}.
\end{proof}

\begin{fact}[Fernandes-Moreno-Rinot, \cite{FMR20} Lemma 4.10 and Proposition 4.14]\label{forcing-comp-rels}
There exists a $<\kappa$-closed $\kappa^+$-cc forcing extension in which $\dl^*_S(\Pi^1_2)$ holds.
\end{fact}

\begin{cor}
There exists a $<\kappa$-closed $\kappa^+$-cc forcing extension in which for all countable complete unsuperstable theory $T$, $\cong_T$ is $\Sigma^1_1$-complete. 
\end{cor}

\section*{Acknowledgments}
This research was partially supported by the European Research Council (grant agreement ERC-2018-StG 802756). This research was partially supported by the Vilho, Yrj\"o and Kalle V\"ais\"al\"a Foundation of the Finnish Academy of Science and Letters. This research was partially supported by the the Austrian Science Fund FWF, grant I 3709-N35 and grant M 3210-N.
 The author wants to express his gratitude Tapani Hyttinen for suggesting the topic during the Ph.D. studies of the author, his advice and fruitful discussions. The author wants to express his gratitude to Gabriel Fernandes and Assaf Rinot for the fruitful discussions on trees.

\providecommand{\bysame}{\leavevmode\hbox to3em{\hrulefill}\thinspace}
\providecommand{\MR}{\relax\ifhmode\unskip\space\fi MR }
\providecommand{\MRhref}[2]{%
  \href{http://www.ams.org/mathscinet-getitem?mr=#1}{#2} }
\providecommand{\href}[2]{#2}


\begin{thebibliography}{Av}

\bibitem{Av}
U.~Abraham, \emph{Construction of a rigid Aronszajn tree}, Proceeding of the American Mathematical Society, \textbf{77}, 136--137 (1979).

\bibitem{FHK13}
S.D. Friedman, T.~Hyttinen, and V.~Kulikov, \emph{Generalized descriptive set theory and classification theory}, in {\it Memories of the American Mathematical Society} {\bf 230} (2014).

\bibitem{FMR20}
G.~Fernandes, M.~Moreno, and A.~Rinot, \emph{Fake reflection}, Israel Journal of Mathematics,  (2021). https://doi.org/10.1007/s11856-021-2213-2

\bibitem{FMR}
G.~Fernandes, M.~Moreno, and A.~Rinot, \emph{Inclusion modulo nonstationary}, Monatshefte für Mathematik, \textbf{192}, 827--851 (2020).

\bibitem{HK}
T.~Hyttinen, and V.~Kulikov, \emph{On $\Sigma^1_1$-complete equivalence
  relations on the generalized Baire space}, Math. Log. Quart.
  \textbf{61}, 66 -- 81 (2015).
  
\bibitem{HKM}
T.~Hyttinen, V.~Kulikov, and M.~Moreno, \emph{A generalized Borel-reducibility counterpart of Shelah's main gap theorem}, Arch. math. Logic.
  \textbf{56}, 175 -- 185 (2017).  

\bibitem{HM}
T. Hyttinen, and M. Moreno, \emph{On the reducibility of isomorphism relations}, Math Logic Quart. \textbf{63}, 175--185 (2017).

\bibitem{HT}
T. Hyttinen, and H. Tuuri, \emph{Constructing strongly equivalent nonisomorphic models for unstable theories}, Annals of Pure and Applied Logic. \textbf{52}, 203--248 (1991).

\bibitem{Mor}
M. Moreno, \emph{The isomorphism relation of theories with S-DOP in the generalized Baire spaces}, Annals of Pure and Applied Logic. \textbf{173}, 103044 (2022). https://doi.org/10.1016/j.apal.2021.103044

\bibitem{Sh1} S.~Shelah, \emph{Diamonds}, Proc. Am. Math. Soc. \textbf{138}, 2151--2161 (2010).

\bibitem{Sh} S.~Shelah, \emph{Existence of many $L_{\infty,\lambda}$-equivalent non-isomorphic models of $T$ of power $\lambda$}, Annals of Pure and Applied Logic \textbf{34}, 291--310 (1987).

\bibitem{Sh90} S.~Shelah, \emph{Classification theory}, Stud. Logic Found. Math. 92, North-Holland 1990.
\end{thebibliography}
\end{document}